\documentclass[a4paper,11pt, reqno]{amsart}

\usepackage{amssymb}
\usepackage{amsmath}
\usepackage{amstext}
\usepackage{amscd}
\usepackage{amsthm}
\usepackage{amsfonts}
\usepackage{mathtools}
\usepackage{array}
\usepackage{color}
\usepackage[all]{xy}
\usepackage{bm}
\usepackage{caption}

\usepackage{tikz}
\usetikzlibrary{arrows}
\usetikzlibrary{snakes}
\usetikzlibrary{shapes}
\usetikzlibrary{calc,decorations.markings}

\usepackage{hyperref}

\makeatletter
  
  \@addtoreset{equation}{section}
\makeatother

\newtheorem{thm}{Theorem}[section]
\newtheorem{lem}[thm]{Lemma}
\newtheorem{prop}[thm]{Proposition}
\newtheorem{cor}[thm]{Corollary}

\newtheorem{que}[thm]{Question}

\theoremstyle{definition}
\newtheorem{dfn}[thm]{Definition}
\newtheorem{ex}[thm]{Example}

\theoremstyle{remark}
\newtheorem{rem}[thm]{Remark}

\newcommand{\rmc}{\mathrm{c}}

\newcommand{\rme}{\mathrm{e}}

\newcommand{\rmH}{\mathrm{H}}

\newcommand{\calH}{\mathcal{H}}

\newcommand{\calO}{\mathcal{O}}

\newcommand{\fkm}{\mathfrak{m}}

\newcommand{\ZZ}{\mathbb{Z}}

\newcommand{\sfN}{\mathsf{N}}
\newcommand{\sfC}{\mathsf{C}}

\newcommand{\ttI}{\mathtt{I}}

\newcommand{\Spec}{\operatorname{Spec}}

\newcommand{\depth}{\operatorname{depth}}

\newcommand{\rank}{\operatorname{rank}}

\newcommand{\add}{\operatorname{add}}

\newcommand{\tor}{\operatorname{tor}}
\newcommand{\Hom}{\operatorname{Hom}}

\newcommand{\SL}{\operatorname{SL}}

\newcommand{\CM}{\operatorname{CM}}

\begin{document}

\title[Ulrich modules over cyclic quotient surface singularities]{Ulrich modules over cyclic quotient surface singularities}
\author[Yusuke Nakajima \and Ken-ichi Yoshida]{Yusuke Nakajima \and Ken-ichi Yoshida}
\date{}


\address[Yusuke Nakajima]{Graduate School Of Mathematics, Nagoya University, Chikusa-Ku, Nagoya,
 464-8602 Japan} 
\email{m06022z@math.nagoya-u.ac.jp}

\address[Ken-ichi Yoshida]{Department of Mathematics, College of Humanities and Sciences, Nihon University, 3-25-40 Sakurajosui, Setagaya-Ku, Tokyo 156-8550, Japan} 
\email{yoshida@math.chs.nihon-u.ac.jp}
\maketitle

\begin{abstract} 
In this paper, we characterize Ulrich modules over cyclic quotient surface singularities 
using the notion of special Cohen-Macaulay modules. 
We also investigate the number of indecomposable Ulrich modules for a given cyclic quotient surface singularity, 
and show that the number of exceptional curves in the minimal resolution determines a boundary on the number of indecomposable Ulrich modules. 
\end{abstract}


\section{Introduction}

Let $(R, \fkm, \Bbbk)$ be a Cohen-Macaulay (= CM) local ring, with $\dim\,R=d$. 
For a finitely generated $R$-module $M$, we say that $M$ is a maximal Cohen-Macaulay (= MCM) $R$-module if $\depth_RM=d$. 
For each MCM $R$-module $M$, we have that $\mu_R(M)\le \rme_\fkm(M)$, where $\mu_R(M)$ denotes the number of minimal generators 
(i.e., $\mu_R(M)=\dim_\Bbbk M/\fkm M$), and $\rme_\fkm(M)$ is the multiplicity of $M$ with respect to $\fkm$. 
Note that if $R$ is a domain, then we have that $\rme_\fkm(M)=(\rank_RM)\rme_\fkm(R)$. 

An Ulrich module is defined as a module that has the maximum number of generators with respect to the above inequality. 
We sometimes call this a maximally generated maximal Cohen-Macaulay module, in line with the original terminology \cite{Ulr, BHU}. 
The name ``Ulrich modules" was introduced in \cite{HK}. 
We remark that the conditions below are inherited by direct summands and direct sums, 
and hence Ulrich modules are closed under direct summands and direct sums.

\begin{dfn} [\cite{Ulr, BHU}] 
\label{def_Ulrich}
Let $M$ be an MCM $R$-module. We say that $M$ is an Ulrich module if it satisfies $\mu_R(M)= \rme_\fkm(M)$. 
\end{dfn}

Several properties of these modules have been investigated in the aforementioned references. 
In a more geometric setting, they have been studied as Ulrich bundles, for example in \cite{ESW, CH1, CH2, CKM}. 
Recently, this notion was generalized for each non-parameter $\fkm$-primary ideal $I$ in \cite{GOTWY1}, 
and this notion has been actively studied (cf. \cite{GOTWY2, GOTWY3}). 
Namely, we say that an MCM $R$-module $M$ is an Ulrich module ``\,with respect to $I$\," if it satisfies the following conditions:
\[
 (1)\, \rme_I(M)=\ell_R(M/IM),  \quad (2)\, M/IM \text{ is an $R/I$-free module},
\]
where $\rme_I(M)$ is the multiplicity of $M$ with respect to $I$, and $\ell_R(M/IM)$ denotes the length of $M/IM$. 
Thus, an Ulrich module with respect to $\fkm$ is nothing else but an Ulrich module in the sense of Definition~\ref{def_Ulrich}. 
(The condition $(2)$ is automatically satisfied if $I=\fkm$.) 
In addition, Ulrich modules have appeared in an attempt to formulate the notion of ``almost Gorenstein rings" \cite{GTT}. 
Thus, it has become more important to understand these modules. 
However, even the existence of an Ulrich module for a given CM local ring is still not known in general. 
Another important problem is to characterize (and classify) Ulrich modules when a given ring $R$ admits an Ulrich module. 
For example, we know the existence of such a module for the following cases: 
\begin{itemize}
\item[$\cdot$] A two dimensional domain with an infinite field \cite{BHU}. 
\item[$\cdot$] A CM local ring that has maximal embedding dimension \cite{BHU}. 
\item[$\cdot$] A strict complete intersection \cite{HUB}. 
\item[$\cdot$] A Veronese subring of a polynomial ring over a field of characteristic $0$ \cite{ESW}. 
\end{itemize}
The characterization problem has also not been solved in many cases. 
Therefore, in this paper we will characterize Ulrich modules (with respect to $\fkm$) over cyclic quotient surface singularities. 
We remark that this singularity is of finite CM representation type (i.e., it has only finitely many non-isomorphic indecomposable MCM modules). 
Since the number of indecomposable Ulrich modules is finite, we will also consider the number of them. 
The key point is to consider special CM modules (see Definition~\ref{def_special}). 
This is another class of MCM modules, and is closely related with the minimal resolution of a quotient surface singularity 
(see Theorem~\ref{special_mckay}). 
Typically, the number of minimal generators of a special CM module is small. 
Thus, special CM modules opposed to Ulrich modules in this sense, 
but these provide us with a simple description of Ulrich modules as follows. 
(For further details regarding terminologies, see later sections.) 

\medskip

Let $R$ be the invariant subring of $\Bbbk[[x, y]]$ under the action of a cyclic group $\frac{1}{n}(1, a)$. 
Then $ M_t\coloneqq\Bigl<x^iy^j\;\Big|\;i+ja\equiv t\;\;(\mathrm{mod}\;n)\Bigl>$ is an MCM $R$-module for $t=0,1,\cdots,n-1$, 
and these cover all the indecomposable MCM modules. 
Suppose that $M_{i_1}, \cdots, M_{i_r}$ are non-free indecomposable special CM $R$-modules ($i_1>\cdots>i_r$). 
We call the subscripts $(i_1, \cdots, i_r)$ the $i$-series. 
Then, we define integers $(d_{1,t},\cdots,d_{r,t})$ for each subscript $t\in [0, n-1]$ as follows: 
\begin{eqnarray*}
  t=d_{1,t}i_1+h_{1,t},\quad &h_{1,t}\in\mathbb{Z}_{\ge 0},&\quad 0\le h_{1,t}<i_1, \\
  h_{u,t}=d_{u+1,t}i_{u+1}+h_{u+1,t},\quad &h_{u+1,t}\in\mathbb{Z}_{\ge 0},&\quad 0\le h_{u+1,t}<i_{u+1},\quad(u=1,\cdots,r-1), \\
  &h_{r,t}=0.&
\end{eqnarray*}
Then, we can describe $t$ as $t=d_{1,t}i_1+d_{2,t}i_2+\cdots+d_{r,t}i_r $, and obtain the following. 

\begin{thm} [= Theorem~\ref{main} and Corollary~\ref{main_cor}]
$M_t$ is an Ulrich $R$-module if and only if $d_{1,t}+d_{2,t}+\cdots+d_{r,t}=\rme(R)-1$ 
where $\rme(R)$ is the multiplicity of $R$ with respect to the maximal ideal. 
\end{thm}

We can prove this theorem by combining the Riemann-Roch formula \cite{Kat} with the special McKay correspondence \cite{Wun1, Wun2}, 
and hence this is a reinterpretation of Wunram's results from the viewpoint of Ulrich modules. 
In addition to a proof of this form, we will provide another one based on Auslander-Reiten theory. 
By using this theorem, we can check which MCM $R$-modules are Ulrich. 
However, it is sometimes challenging to compute $(d_{1,t},\cdots,d_{r,t})$ for all $t=0, 1, \cdots, n-1$.  
Thus, we provide a sharper characterization of Ulrich modules in terms of the $i$-series as follows. 
The crucial point is to consider good sequences of the $i$-series. 

\begin{thm} [= Theorem~\ref{main2} and \ref{main3}]
Take any sequences of the $i$-series $i_{k(1)}>i_{k(2)}>\cdots>i_{k(2b)}$ with $i_{k(2c-1)}\in\ttI_{n-1}$ (see (\ref{def_In})), for all $c=1,\cdots,b$. 
If $t=n-1-\displaystyle\sum^b_{c=1}\big(i_{k(2c-1)}-i_{k(2c)}\big)$ or $t=n-1$, then $M_t$ is an Ulrich module. 

Conversely, if $M_t$ is an Ulrich module $(t\neq n-1)$, then there exists a sequence 
of the $i$-series $i_{k(1)}>i_{k(2)}>\cdots>i_{k(2b)}$ with $i_{k(2c-1)}\in\ttI_{n-1}$ for all $c=1,\cdots,b$ and 
\[
t=n-1-\displaystyle\sum^b_{c=1}\big(i_{k(2c-1)}-i_{k(2c)}\big).
\]
\end{thm}

As a corollary, we obtain the following. 

\begin{cor}[= Corollary~\ref{main_cor2}] 
We suppose that $M_t$ is an Ulrich module. 
Then, we have that $n-a\le t\le n-1$. 
Furthermore, $M_{n-1}$ and $M_{n-a}$ are certainly Ulrich modules. 
\end{cor}

In particular, this result gives us an upper bound on the number of Ulrich modules. 
Namely, this number is less than or equal to $a$. 
Furthermore, we can obtain additional bounds from further geometric information.  

\begin{thm}[= Theorem~\ref{thm_upper}]
Suppose that $R$ is a cyclic quotient surface singularity, whose number of  irreducible exceptional curves 
$($= that of non-free indecomposable special CM modules$)$ is $r$. 
Then, the number of Ulrich modules $\sfN$ satisfies $r\le\sfN\le 2^{r-1}$.  
\end{thm} 


The remainder of this paper is organized as follows. 
First, we introduce the notion of special CM modules in Section~\ref{cyclic_case}. 
In addition, we introduce the Auslander-Reiten quiver. 
This oriented graph visualizes the relations between MCM modules, and plays an important role in characterizing Ulrich modules. 
For cyclic cases, special CM modules are characterized using purely combinatorial data. 
Using such data, we describe characterizations of Ulrich modules over 
cyclic quotient surface singularities in Section~\ref{main_sec}. 
In Section~\ref{further_topics}, we consider related topics. In particular, we investigate the number of indecomposable Ulrich modules. 

\subsection*{Notations}
Throughout this paper, we assume that $\Bbbk$ is an algebraically closed field. 
Since our focus is on Ulrich modules (with respect to $\fkm$), we employ the notation $\rme(M)\coloneqq \rme_\fkm(M)$, for simplicity. 
We denote the $R$-dual (resp. the canonical dual) functor by $(-)^*\coloneqq\Hom_R(-,R)$ (resp. $(-)^\dagger\coloneqq\Hom_R(-,\omega_R)$).  
Furthermore, we denote the first syzygy functor by $\Omega(-)$. 
We denote the category of MCM modules by $\CM(R)$, and 
the full subcategory consisting of direct summands of finite direct sums of copies of $M$ by $\add_R(M)$. 

\section{Preliminaries}
\label{cyclic_case} 

In this section, we review some known results concerning cyclic quotient surface singularities. 
Thus, we suppose that $G$ is a cyclic group: 
\[
G\coloneqq \langle\;\sigma=
    \begin{pmatrix} \zeta_n&0 \\
                    0&\zeta_n^a 
    \end{pmatrix}           \;\rangle,
\]
where $\zeta_n$ is a primitive $n$-th root of unity with $1\le a\le n-1$, and $\mathrm{gcd}(a,n)=1$. 
We assume that $n$ is invertible in $\Bbbk$. We denote the cyclic group $G$ by $\frac{1}{n}(1,a)$.
Let $S\coloneqq \Bbbk[[x,y]]$ be a power series ring. We denote the invariant subring of $S$ under the action of $G$ by $R\coloneqq S^G$. 
Since $G$ is an abelian group, every irreducible representation of $G$ is one dimensional, and can be described as 
\[
 V_t:\sigma\mapsto\zeta_n^{-t} \quad(t=0,1,\cdots,n-1).
\]
We define 
\[
 M_t\coloneqq(S\otimes_\Bbbk V_t)^G=\Bigl<x^iy^j\;\Big|\;i+ja\equiv t\;\;(\mathrm{mod}\;n)\Bigl>,\;\;(t=0,1,\cdots,n-1).
\]
Then, each $M_t$ is an MCM $R$-module, and $M_s\not\cong M_t$ if $s\neq t$. 
It is well known that $R$ is of finite CM representation type, 
and the modules $M_t$ define all indecomposable MCM modules over $R$, and $\rank_R M_t=1$.

\subsection{Special Cohen-Macaulay modules}
Next, we introduce the notion of special CM modules. 
As we will see in Theorem~\ref{cyclic_special}, every non-free special CM module is minimally two-generated. 
Thus, special CM modules represent the opposite case to Ulrich modules concerning the number of minimal generators. 
However, they will play a crucial role when we characterize Ulrich modules. Here, we recall the definition of special CM modules.

\begin{dfn}[\cite{Wun2}]
\label{def_special}
For an MCM $R$-module $M$, we say that $M$ is special if $(M\otimes_R\omega_R)\big/\tor$ is also an MCM $R$-module.
\end{dfn} 

We will see in Theorem~\ref{special_mckay} that there is a one-to-one correspondence between 
non-free indecomposable special CM $R$-modules and irreducible exceptional curves in the minimal resolution of $\Spec R$. 
Moreover, this constitutes a generalization of the classical McKay correspondence, because every MCM module is special 
if $R$ is Gorenstein (i.e., $G\subset\SL(2,\Bbbk)$ \cite{Wat}). 
Furthermore, there exists a further characterization of special CM modules, as we shall see in Proposition~\ref{char_special}. 
For further details regarding special CM modules, we refer the reader to the references \cite{Wun1, Wun2, Ish, Ito, IW, Rie}. 

\begin{prop}
\label{char_special}
{\rm $($\cite[2.7 and 3.6]{IW}$)$}
Suppose that $M$ is an MCM $R$-module. 
Then, $M$ is a special CM module if and only if it satisfies $(\Omega M)^*\cong M$.
\end{prop}
 
For a cyclic group $G=\frac{1}{n}(1,a)$, we can determine special CM modules by using the following combinatorial data. 
For the first step, we consider the Hirzebruch-Jung continued fraction expansion of $n/a$: 
\[
 \frac{n}{a}=\alpha_1-\cfrac{1}{\alpha_2-\cfrac{1}{\cdots -\frac{1}{\alpha_r}}}\coloneqq[\alpha_1,\alpha_2,\cdots,\alpha_r].
\]
Then, we define the notion of the $i$-series and $j$-series (see \cite{Wem, Wun1}).

\begin{dfn}
For $n/a=[\alpha_1,\alpha_2,\cdots,\alpha_r]$, the $i$-series and $j$-series are defined as follows:
\begin{eqnarray*}
 i_0=n,\;\;i_1=a,\;\;&i_t=\alpha_{t-1}i_{t-1}-i_{t-2}&\;\;(t=2,\cdots,r+1), \\
 j_0=0,\;\;j_1=1,\;\;&j_t=\alpha_{t-1}j_{t-1}-j_{t-2}&\;\;(t=2,\cdots,r+1).
\end{eqnarray*}
\end{dfn}

By using the $i$-series and $j$-series, we can characterize special CM $R$-modules.

\begin{thm}[\cite{Wun1}]
\label{cyclic_special}
For a cyclic group $G=\frac{1}{n}(1,a)$ with $n/a=[\alpha_1,\alpha_2,\cdots,\alpha_r]$, 
$M_{i_t}\;(t=1,\cdots,r)$ and $R$ are precisely special CM modules over $R$. 
Furthermore, the minimal generators of $M_{i_t}$ are $x^{i_t}$ and $y^{j_t}$, for $t=1,\cdots,r$. 
\end{thm}

As we mentioned above, special CM modules are compatible with the geometrical structures. 
Thus, we introduce some terminologies on the geometric side, and discuss a relationship between   
special CM modules and geometrical objects. 

Let $\pi:X\rightarrow\Spec\,R$ be the minimal resolution of singularities, and $E\coloneqq\pi^{-1}(\fkm)$ be the exceptional divisor. 
We decompose $E=\bigcup^r_{t=1}E_{i_t}$ into irreducible components, 
and define a cycle $Z=\sum^r_{t=1}a_{i_t}E_{i_t}$ with $a_{i_t}\in\ZZ$. 
For cycles $Z, Z^\prime$, we denote the intersection number of $Z$ and $Z^\prime$ by $Z\cdot Z^\prime$. 
If $Z=Z^\prime$, then the self-intersection number of $Z$ is denoted by $Z^2$.  
We say that a cycle $Z$ is positive if $a_{i_t}>0$ for all $i_t$, 
and say that a positive cycle $Z$ is anti-nef if $Z\cdot E_{i_t}\le 0$ for all $i_t$. 
We define the fundamental cycle $Z_0$ as the unique smallest element in the set of anti-nef cycles. 
An algorithm exists to determine $Z_0$ (see \cite{Lau}), and the fundamental cycle is $Z_0=\sum^r_{t=1}E_{i_t}$ in our situation. 

\medskip

The following is a famous result known as the special McKay correspondence. 

\begin{thm}[\cite{Wun2}]
\label{special_mckay}
For any $i_t$, there exists a unique indecomposable MCM $R$-module $M_{i_t}$ $($up to isomorphism$)$ 
such that $\rmH^1(\widetilde{M_{i_t}}^\vee)=0$ and $\rmc_1(\widetilde{M_{i_t}})\cdot E_{i_s}=\delta_{st}$ for $1\le s, t  \le r$, 
where $\widetilde{M_{i_t}}=\pi^*(M_{i_t})/\tor$, $\rmc_1(\widetilde{M_{i_t}})$ denotes the first Chern class of $\widetilde{M_{i_t}}$, 
and $(-)^\vee=\calH\!om_{\calO_X}(-, \calO_X)$. 
These MCM modules $M_{i_1}, \cdots, M_{i_r}$ are precisely indecomposable non-free special CM modules, 
and $\rank_RM_{i_t}=\rmc_1(\widetilde{M_{i_t}})\cdot Z_0$. 
\end{thm}

By this theorem, there exists a one-to-one correspondence between non-free indecomposable special CM modules and 
irreducible exceptional curves. 
The dual graph of the minimal resolution of singularity $X\rightarrow\Spec(R)$ can also be obtained by the Hirzebruch-Jung continued fraction expansion: 
\[\scalebox{0.8}{
\begin{tikzpicture}
\node (A1) at (1,0){$-\alpha_1$};
\node (A2) at (3,0){$-\alpha_2$}; 
\node (Ar) at (7,0){$-\alpha_r$}; 

\node (E1) at (1,0.7){$E_{i_1}$};
\node (E2) at (3,0.7){$E_{i_2}$}; 
\node (Er) at (7,0.7){$E_{i_r}$}; 

\draw  [thick] (A1) circle [radius=0.38] ;
\draw  [thick] (A2) circle [radius=0.38] ;
\draw  [thick] (Ar) circle [radius=0.38] ;
\draw (A1)--(A2) ; \draw (A2)--(4.3,0) ; \draw[thick, dotted] (4.5,0)--(5.5,0); \draw (5.7,0)--(Ar) ;
\end{tikzpicture}
}\]
where each circled number represents the self-intersection number of the corresponding exceptional curve. 

When we consider Ulrich modules, the multiplicity $\rme(M)=(\rank_RM)\rme(R)$ is important. 
It is known that the multiplicity $\rme(R)$ can be computed via the self-intersection number of the fundamental cycle $Z_0$ \cite{Art}. 
That is, we have that $\rme(R)=-Z_0^2=\alpha_1+\cdots+\alpha_r-2(r-1)$.

\begin{ex}
\label{ex12/7}
Let $G=\frac{1}{12}(1,7)$ be a cyclic group of order $12$. 
The Hirzebruch-Jung continued fraction expansion of $12/7$ is 
\[
 \frac{12}{7}=2-\cfrac{1}{4-1/2}=[2,4,2]. 
\]
The $i$-series and the $j$-series are given by
\[
\begin{array}{ccccc}
 i_0=12,&i_1=7,&i_2=2,&i_3=1,&i_4=0, \\
 j_0=0,&j_1=1,&j_2=2,&j_3=7,&j_4=12.
\end{array}
\]
Thus, the special CM modules are $M_7, \,M_2, \,M_1, \,R$, and these take the form  
\[
  M_7=Rx^7+Ry, \quad M_2=Rx^2+Ry^2, \quad  M_1=Rx+Ry^7 .
\]
In this case, the dual graph is 
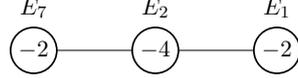
\[\scalebox{0.8}{
\begin{tikzpicture}
\node (A1) at (1,0){$-2$};
\node (A2) at (3,0){$-4$}; 
\node (A3) at (5,0){$-2$}; 

\node (E1) at (1,0.7){$E_7$};
\node (E2) at (3,0.7){$E_2$}; 
\node (E3) at (5,0.7){$E_1$}; 

\draw  [thick] (A1) circle [radius=0.38] ;
\draw  [thick] (A2) circle [radius=0.38] ;
\draw  [thick] (A3) circle [radius=0.38] ;
\draw (A1)--(A2) ; \draw (A2)--(A3) ;
\end{tikzpicture}
}\]
and the fundamental cycle is $Z_0=E_7+E_2+E_1$. 
Thus, we have the multiplicity $\rme(R)=-Z_0^2=4$. 
\end{ex}

\subsection{Auslander-Reiten theory}
In the previous subsection, we saw that the multiplicity can be computed using the fundamental cycle. 
To determine Ulrich modules, we will also investigate the number of minimal generators. 
In Section~\ref{main_sec}, we will use the Auslander-Reiten quiver to understand minimal generators. 
Moreover, by applying the functor $\tau$, which is called the Auslander-Reiten translation, to a special CM module, 
we can obtain an Ulrich module (see Proposition~\ref{Ulrich_Rdual}). 
Thus, in this subsection we will present some results from Auslander-Reiten theory. 
Although we will mainly discuss the case of a cyclic quotient surface singularity $R$, 
we can obtain similar results for any quotient surface singularities. 
For more details, see \cite{LW,Yo}. 

\begin{dfn}
Let $M$ and $N$ be indecomposable MCM $R$-modules.
We call a non-split short exact sequence 
$0\rightarrow N\overset{f}{\rightarrow}L\overset{g}{\rightarrow}M\rightarrow 0$ 
the Auslander-Reiten (= AR) sequence ending in $M$ if for any MCM module $X$ and any morphism $\varphi:X\rightarrow M$ 
that is not a split surjection, there exists $\phi:X\rightarrow L$ such that $\varphi=g\circ\phi$.
\end{dfn}

Since $R$ is an isolated singularity, the AR sequence ending in $M$ exists 
for any non-free indecomposable MCM $R$-module $M$, and is unique up to isomorphism \cite{Aus2}. 
We can construct the AR sequence by using the Koszul complex over $S$ and a natural representation of $G$ 
(see, e.g., \cite[Chapter~10]{Yo}). 
In our situation, the following is the AR sequence ending in $M_t\;(t\neq 0)$: 
\begin{equation}
\label{AR_M}
0\longrightarrow M_{t-a-1}\longrightarrow M_{t-1}\oplus M_{t-a}\longrightarrow M_t\longrightarrow 0.
\end{equation}
For the case of $t=0$, we also have the fundamental sequence of $R$: 
\begin{equation}
\label{fund_R}
0\longrightarrow\omega_R\longrightarrow M_{-1}\oplus M_{-a}\longrightarrow R\longrightarrow \Bbbk\longrightarrow 0.
\end{equation}
We call the left term of the AR sequence ending in $M_t$ the Auslander-Reiten (AR) translation of $M_t$, and denote this by $\tau(M_t)$. 
It is known that the AR translation $\tau$ can be obtained via the functors 
\[
\tau: \CM(R)\overset{(-)^*}{\longrightarrow}\CM(R)\overset{(-)^\dagger}{\longrightarrow}\CM(R). 
\]
Furthermore, we denote the middle term of an AR sequence by $E_{M_t}$. 
Thus, in our situation we have that $E_{M_t}=M_{t-1}\oplus M_{t-a}$ and $\tau(M_t)=M_{t-a-1}$ for $t=0,1,\cdots,n-1$. 

\medskip

Next, we introduce the notion of the Auslander-Reiten quiver. 

\begin{dfn}
The Auslander-Reiten (= AR) quiver of $R$ is an oriented graph whose vertices are indecomposable MCM $R$-modules $R,M_1,\cdots,M_{n-1}$, 
where we draw $m_{st}$ arrows from $M_s$ to $M_t$ $(s,t=0,1,\cdots,n-1)$, with $m_{st}$ denoting the multiplicity of $M_s$ in the decomposition of $E_{M_t}$. 
\end{dfn}

Since we already know that $E_{M_t}=M_{t-1}\oplus M_{t-a}$, 
it follows that we draw an arrow from $M_{t-1}$ to $M_t$ and from $M_{t-a}$ to $M_t$ for $t=0,1,\cdots,n-1$. 
Furthermore, we see that these arrows correspond to morphisms $\cdot x$ and $\cdot y$, respectively:
\[
M_{t-1}=\big\{\;f\in S\;|\;\sigma\cdot f=\zeta_n^{t-1}\;f\;\big\}\overset{\cdot x}{\longrightarrow}
M_t=\big\{\;f\in S\;|\;\sigma\cdot f=\zeta_n^t\;f\;\big\} 
\]
\[
M_{t-a}=\big\{\;f\in S\;|\;\sigma\cdot f=\zeta_n^{t-a}\;f\;\big\}\overset{\cdot y}{\longrightarrow}
M_t=\big\{\;f\in S\;|\;\sigma\cdot f=\zeta_n^t\;f\;\big\}. 
\] 
Moreover, the AR quiver of $R$ coincides with the McKay quiver of $G$ \cite{Aus1}.

\begin{ex}
\label{ex12/7_AR}
Let $G=\frac{1}{12}(1,7)$ be a cyclic group of order $12$ (see Example~\ref{ex12/7}). 
The AR quiver of $R=S^G$ is the following. For simplicity, we only describe subscripts as vertices. 

\medskip

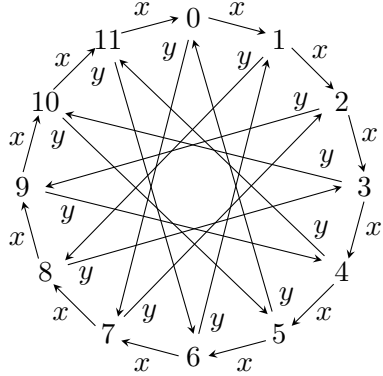
\begin{center}
\begin{tikzpicture}[>=stealth, scale=.75]
\node (P0) at (90:3cm) {$0$}; 
\node (P1) at (60:3cm) {$1$}; 
\node (P2) at (30:3cm) {$2$}; 
\node (P3) at (0:3cm) {$3$}; 
\node (P4) at (-30:3cm) {$4$}; 
\node (P5) at (-60:3cm) {$5$}; 
\node (P6) at (-90:3cm) {$6$}; 
\node (P7) at (-120:3cm) {$7$}; 
\node (P8) at (-150:3cm) {$8$}; 
\node (P9) at (180:3cm) {$9$}; 
\node (P10) at (150:3cm) {$10$}; 
\node (P11) at (120:3cm) {$11$}; 

\draw[->] (85:3cm)--node [swap,above]{$x$} (65:3cm); 
\draw[->] (55:3cm)--node [swap,above,xshift=1mm] {$x$} (35:3cm); 
\draw[->] (25:3cm)--node [swap,above,xshift=1.8mm, yshift=-1mm] {$x$} (5:3cm); 
\draw[->] (-5:3cm)--node [swap,above,xshift=2.3mm, yshift=-2mm] {$x$} (-25:3cm);
\draw[->] (-35:3cm)--node [swap,below,xshift=1.5mm, yshift=1mm] {$x$} (-55:3cm);  
\draw[->] (-65:3cm)--node [swap,below,xshift=1mm] {$x$} (-85:3cm); 
\draw[->] (-95:3cm)--node [swap,below,xshift=-1mm] {$x$} (-115:3cm); 
\draw[->] (-125:3cm)--node [swap,below,xshift=-1.8mm, yshift=1mm] {$x$} (-145:3cm); 
\draw[->] (-155:3cm)--node [swap,below,xshift=-1.8mm, yshift=1mm] {$x$} (185:3cm); 
\draw[->] (175:3cm)--node [swap,above,xshift=-1.8mm, yshift=-1.5mm] {$x$} (155:3cm); 
\draw[->] (145:3cm)--node [swap,above,xshift=-1.5mm, yshift=-1mm] {$x$} (125:3cm); 
\draw[->] (115:3cm)--node [swap,above,xshift=-1mm] {$x$} (95:3cm); 

\draw[->] (92:2.6cm)--node [swap,above,xshift=1.2cm, yshift=1.3cm] {$y$} (-120:2.65cm); 
\draw[->] (62:2.6cm)--node [swap,above,xshift=1.8cm, yshift=0.5cm] {$y$}  (-150:2.6cm); 
\draw[->] (32:2.6cm)--node [swap,above,xshift=1.9cm, yshift=-0.4cm] {$y$}  (180:2.6cm); 
\draw[->] (2:2.6cm)--node [swap,above,xshift=1.55cm, yshift=-1.35cm] {$y$}  (150:2.6cm); 
\draw[->] (-28:2.6cm)--node [swap,above,xshift=0.85cm, yshift=-2.1cm] {$y$}  (120:2.6cm); 
\draw[->] (-58:2.6cm)--node [swap,above,xshift=-0.2cm, yshift=-2.2cm] {$y$}  (90:2.6cm); 
\draw[->] (-88:2.6cm)--node [swap,below,xshift=-1.1cm, yshift=-1.3cm] {$y$} (60:2.6cm); 
\draw[->] (-118:2.6cm)--node [swap,below,xshift=-1.8cm, yshift=-0.55cm] {$y$} (30:2.6cm); 
\draw[->] (-148:2.6cm)--node [swap,below,xshift=-1.8cm, yshift=0.4cm] {$y$} (0:2.6cm); 
\draw[->] (182:2.6cm)--node [swap,below,xshift=-1.65cm, yshift=1.4cm] {$y$} (-30:2.6cm); 
\draw[->] (152:2.6cm)--node [swap,below,xshift=-0.88cm, yshift=2.1cm] {$y$} (-60:2.6cm); 
\draw[->] (122:2.6cm)--node [swap,below,xshift=0.2cm, yshift=2.2cm] {$y$} (-90:2.6cm); 

\end{tikzpicture}
\end{center}

\end{ex}

\section{Ulrich modules for cyclic cases}
\label{main_sec}

In this section, we will present a characterization of Ulrich modules using special CM modules. 
First, we note the following proposition. 

\begin{prop}
\label{Ulrich_Rdual}
Let the notation be the same as in Section~\ref{cyclic_case}. 
For a non-free special CM $R$-module $M_{i_t}$, the MCM modules $M_{n-i_t}$ and $M_{i_t-a-1}$ are Ulrich modules.  
\end{prop}

\begin{proof}
By Proposition~\ref{char_special}, $M_{i_t}^*$ is the syzygy of an MCM $R$-module. 
Thus, we can see that it is an Ulrich module by a similar argument as in \cite[Lemma~4.2]{GOTWY1}. 
In addition, the canonical dual of an Ulrich module is an Ulrich module, by \cite[Corollary~1.4]{Ooi}. 
Thus, $\tau(M_{i_t})=(M_{i_t}^*)^\dagger$ is also an Ulrich module. 
Since $M_{i_t}^*\cong M_{n-i_t}$ and $\tau(M_{i_t})\cong M_{i_t-a-1}$, we have the desired conclusion. 
\end{proof}

By this proposition, we can obtain some examples of Ulrich modules. 
However, there exist Ulrich modules that do not take the form given in Proposition~\ref{Ulrich_Rdual}. 
In order to determine all of these, we will describe the relationship between the multiplicity $\rme(M_t)=\rme(R)$ 
and the number of minimal generators $\mu_R(M_t)$ in terms of the $i$-series. 
To state our theorem, we prepare some notations. 
For the $i$-series $(i_1,\cdots,i_r)$ associated with $\frac{1}{n}(1,a)$ and for any $t\in [0, n-1]$,  
there exist unique non-negative integers $d_{1,t},\cdots,d_{r,t}\in\mathbb{Z}_{\ge 0}$ such that 
\begin{eqnarray*}
  t=d_{1,t}i_1+h_{1,t},\quad &h_{1,t}\in\mathbb{Z}_{\ge 0},&\quad 0\le h_{1,t}<i_1, \\
  h_{u,t}=d_{u+1,t}i_{u+1}+h_{u+1,t},\quad &h_{u+1,t}\in\mathbb{Z}_{\ge 0},&\quad 0\le h_{u+1,t}<i_{u+1},\quad(u=1,\cdots,r-1), \\
  &h_{r,t}=0.&
\end{eqnarray*}

Thus, we can describe $t$ as follows: 
\begin{eqnarray*}
  t&=&d_{1,t}i_1+d_{2,t}i_2+\cdots+d_{r,t}i_r \\
  &=&(\underbrace{i_1+\cdots+i_1}_{d_{1,t}})+(\underbrace{i_2+\cdots+i_2}_{d_{2,t}})+\cdots+(\underbrace{i_r+\cdots+i_r}_{d_{r,t}}).
\end{eqnarray*}

We are now in the position to state our theorem.

\begin{thm}
\label{main}
Let the notation be the same as above. Then, we have that
\[
\mu_R(M_t)=d_{1,t}+d_{2,t}+\cdots +d_{r,t}+1. 
\]
\end{thm}

We will present proofs of two forms (one geometric and one representation theoretic) for this theorem. 
The geometric proof is quite simple, and states that the above formula is a reinterpretation of special McKay correspondence, 
from the viewpoint of Ulrich modules. 
However, the authors believe that the method used in the other proof will provide us with new insight to this subject 
(see, e.g., Remark~\ref{non_cyclic}).   
Therefore, we present both of them. 

\begin{proof}[Geometric proof of Theorem~\ref{main}] 
By Kato's Riemann-Roch formula \cite{Kat}, we have that
\[
\mu_R(M_t)=1+\rmc_1(\widetilde{M_t})\cdot(E_{i_1}+\cdots+E_{i_r}). 
\]
Furthermore, $\rmc_1(\widetilde{M_t})\cdot E_{i_u}=d_{u,t}$ \cite{Wun2}. Thus, we have reached the desired conclusion. 
\end{proof}

\medskip

Before moving to the representation theoretic proof, we note a crucial observation. 
Since $M_t\cong\Hom_R(R, M_t)$, a path from $R$ to $M_t$ on the AR quiver corresponds to an element of $M_t$. 
For example, in the AR quiver of Example~\ref{ex12/7_AR}, we can find a path 
\[
0\overset{x}{\longrightarrow}1\overset{y}{\longrightarrow}8\overset{x}{\longrightarrow}9\overset{x}{\longrightarrow}10\overset{y}{\longrightarrow}5, 
\] 
and a unit $1\in R$ maps to $x^3y^2\in M_5$ via the above path. 
Note that if a given path from $R$ to $M_t$ factors through a free module that is not the starting point, then its image will be in $\fkm M_t$. 
Thus, by Nakayama's lemma such a path does not correspond to a minimal generator of $M_t$, 
\[
\fkm M_t\cong\{R\overset{\text{non-split}}{\rightarrow}R^{\oplus m}\rightarrow M_t\}. 
\]
Therefore, we can identify a minimal generator of $M_t$ with a path from $R$ to $M_t$ that does not factor through any free module except the starting point.  
In the following, we will count such paths on the AR quiver. 

\begin{proof}[Representation theoretic proof of Theorem~\ref{main}] 
We write the AR quiver $\mathsf{Q}$ in the form of the translation quiver $\ZZ\mathsf{Q}$, as shown in Figure~\ref{proof_AR} 
(where this is the repetition of the AR quiver).   
Here, each diagram 
$\begin{array}{c}\tiny{
\xymatrix@C=8pt@R=8pt{a\ar[r]^y\ar@{}[dr]|\circlearrowleft&b\\c\ar[u]^x\ar[r]_y&d\ar[u]_x}}\end{array}$ 
corresponds to the AR sequence 
$0\rightarrow M_c\rightarrow M_a\oplus M_d\rightarrow M_b\rightarrow 0$ 
ending in $M_b$ ($b\neq 0$) and the fundamental sequence of $R$. In particular, they are each commutative. 

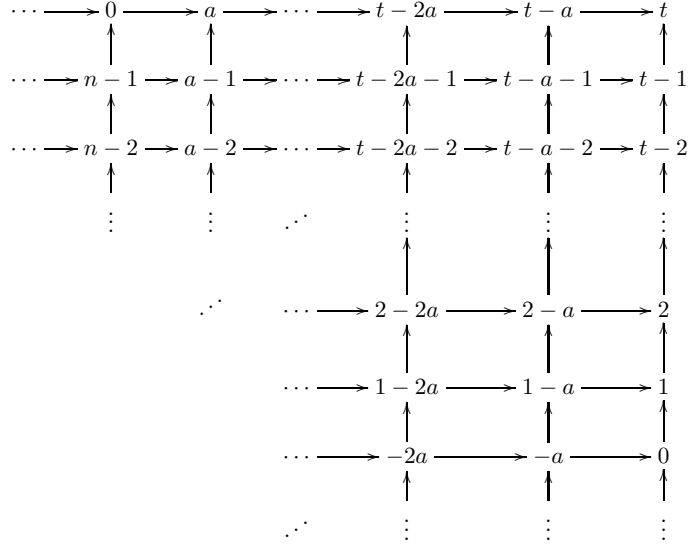
\begin{figure}[!h]
\begin{center}
\scalebox{0.8}{
\xymatrix@C=15pt@R=18pt{
  \cdots\ar[r]&0\ar[r]&a\ar[r]&\cdots\ar[r]&t-2a\ar[r]&t-a\ar[r]&t \\
  \cdots\ar[r]&n-1\ar[u]\ar[r]&a-1\ar[u]\ar[r]&\cdots\ar[r]&t-2a-1\ar[u]\ar[r]&t-a-1\ar[u]\ar[r]&t-1\ar[u] \\
  \cdots\ar[r]&n-2\ar[u]\ar[r]&a-2\ar[u]\ar[r]&\cdots\ar[r]&t-2a-2\ar[u]\ar[r]&t-a-2\ar[u]\ar[r]&t-2\ar[u] \\
  &\vdots\ar[u]&\vdots\ar[u]&\reflectbox{$\ddots$}&\vdots\ar[u]&\vdots\ar[u]&\vdots\ar[u] \\
  &&\reflectbox{$\ddots$}&\cdots\ar[r]&2-2a\ar[u]\ar[r]&2-a\ar[u]\ar[r]&2\ar[u] \\
  &&&\cdots\ar[r]&1-2a\ar[u]\ar[r]&1-a\ar[u]\ar[r]&1\ar[u] \\
  &&&\cdots\ar[r]&-2a\ar[u]\ar[r]&-a\ar[u]\ar[r]&0\ar[u]  \\
  &&&\reflectbox{$\ddots$}&\vdots\ar[u]&\vdots\ar[u]&\vdots\ar[u] 
} }
\end{center}
\caption{The AR quiver with the form of the translation quiver }
\label{proof_AR}
\end{figure}

From this quiver, we extract appropriate paths from $R \ (= 0)$ to $M_t \ (= t)$ corresponding to 
minimal generators of $M_t$. Such paths take the form shown in Figure~\ref{proof_AR1}.  
Here, we assume that grayed areas do not contain $0$, otherwise we can divide those areas into smaller ones. 
Since any $0$ vertex that is located at the outside of Figure~\ref{proof_AR1} certainly goes through a free module on the way to $M_t$, 
we need only consider paths from $R$ to $M_t$ appearing in Figure~\ref{proof_AR1}. 
Furthermore, the number of $0$ vertices appearing in Figure~\ref{proof_AR1} coincides with $\mu_R(M_t)$.  
To clarify the situation, we denote the $i$-th column from the right containing $0$ by $\ell_i$, for $i=1,\cdots,\mu_R(M_t)$. 
We see that the length of $\ell_1$ is $t$, and it is divided into $\mu_R(M_t)-1$ blocks. Furthermore,
we denote vertices located at the corner of a diagram by $\bigstar_1, \bigstar_2, \cdots, \bigstar_{\nu-1},\bigstar_\nu$, 
as shown in Figure~\ref{proof_AR1}. 
We remark that these vertices are special CM modules, because the number of minimal generators for each of these is two. 
From this construction, we have that $\nu=\mu_R(M_t)-1$. 
In the following, we will show that $\nu=d_{1,t}+d_{2,t}+\cdots+d_{r,t}$. (For simplicity, we will simply denote $d_{u, t}$ by $d_u$.) 

\begin{figure}[!h]
\begin{center}
\begin{tikzpicture}[>=stealth, scale=0.65]
\node (P1) at (0,0) {$0$}; 
\node (P2) at (3,0) {$\bigstar_\nu$}; 
\node (P3) at (15,0) {$t$}; 

\node (P4) at (3,-3) {$0$}; 
\node (P5) at (6,-3) {$\bigstar_{\nu-1}$}; 
\node (P6) at (15,-1) {{\footnotesize$t-1$}}; 
\node (P7) at (6,-6) {$0$}; 
\node (P8) at (9,-9) {$0$}; 
\node (P9) at (12,-9) {$\bigstar_2$}; 
\node (P10) at (12,-12) {$0$}; 
\node (P11) at (15,-12) {$\bigstar_1$};
\node (P12) at (15,-15) {$0$};  

\node (L1) at (15,-16.2) {$\ell_1$} ;
\node (L2) at (12,-16.2) {$\ell_2$} ;
\node (L3) at (9,-16.2) {$\ell_3$} ;
\node (Lm-1) at (3,-16.2) {$\ell_{\mu(M_t)-1}$} ;
\node (Lm) at (0,-16.2) {$\ell_{\mu(M_t)}$} ;

\filldraw[draw=black!10, fill=black!10]  (0,-0.3) rectangle (2.7,-3);
\filldraw[draw=black!10, fill=black!10]  (3,-3.3) rectangle (5.7,-6);
\filldraw[draw=black!10, fill=black!10]  (9,-9.3) rectangle (11.7,-12);
\filldraw[draw=black!10, fill=black!10]  (12,-12.3) rectangle (14.7,-15); 
\filldraw[draw=black!10, fill=black!10]  (3.3,-0.3)--(3.3,-2.7)--(6.8,-2.7)--(6.8,-5.7)--(7.3,-5.7)--(9.3,-7.7)--(9.3,-8.7)--(12.5,-8.7)
--(12.5,-11.7)--(14.35,-11.7)--(14.35,-0.3)--(3.3,-0.3);

\draw[->] (P1)--(0.9,0); 
\draw[dotted, thick] (1.1,0)--(1.7,0); 
\draw[->] (1.8,0)--(P2);
\draw[->] (P2)--(3.9,0); 
\draw[dotted, thick] (4.1,0)--(4.9,0); 
\draw[dotted, thick] (8.1,0)--(9.9,0); 
\draw[dotted, thick] (13.1,0)--(13.9,0); 
\draw[->] (14.1,0)--(P3); 

\draw[->] (P4)--(3.9,-3) ;
\draw[dotted, thick] (4.1,-3)--(4.5,-3);
\draw[->] (4.6,-3)--(P5) ; 
\draw[->] (P7)--(6.9,-6) ;
\draw[->] (P8)--(9.9,-9) ;
\draw[dotted, thick] (10.1,-9)--(10.9,-9);
\draw[->] (11.1,-9)--(P9) ;
\draw[->] (P10)--(12.9,-12) ;
\draw[dotted, thick] (13.1,-12)--(13.9,-12);
\draw[->] (14.1,-12)--(P11) ;

\draw[dotted, thick] (7.6,-6.5)--(8.5,-7.4); 

\draw[->] (P6)--(P3); 
\draw[->] (15,-1.9)--(P6); 
\draw[dotted, thick] (15,-2.9)--(15,-2.1);
\draw[dotted, thick] (15,-7.9)--(15,-5.1);
\draw[dotted, thick] (15,-10.9)--(15,-10.1);
\draw[->] (P11)--(15,-11.1) ;

\draw[->] (3,-0.9)--(P2); 
\draw[dotted, thick] (3,-1.9)--(3,-1.1);
\draw[->] (P4)--(3,-2.1); 
\draw[->] (6,-3.9)--(P5); 
\draw[dotted, thick] (6,-4.9)--(6,-4.1);
\draw[->] (P7)--(6,-5.1); 
\draw[->] (P8)--(9,-8.1) ;
\draw[->] (12,-9.9)--(P9); 
\draw[dotted, thick] (12,-10.9)--(12,-10.1);
\draw[->] (P10)--(12,-11.1); 
\draw[->] (15,-12.9)--(P11); 
\draw[dotted, thick] (15,-13.9)--(15,-13.1);
\draw[->] (P12)--(15,-14.1); 

\draw (15.8,0)--(17,0);
\draw (15.8,-3)--(17,-3);
\draw (15.8,-6)--(17,-6);
\draw (15.8,-9)--(17,-9);
\draw (15.8,-12)--(17,-12);
\draw (15.8,-15)--(17,-15);

\draw[<->,thick] (16.4,-0.2)--(16.4,-2.8) ;
\draw[<->,thick] (16.4,-3.2)--(16.4,-5.8) ;
\draw[<->,thick] (16.4,-9.2)--(16.4,-11.8) ;
\draw[<->,thick] (16.4,-12.2)--(16.4,-14.8) ;
\draw[dotted, very thick] (16.4,-7)--(16.4,-8);

\draw (17.5,0) to [out=-75, in=75] node [swap,above, xshift=0.85cm]{{\footnotesize$\mu(M_t)-1$}} (17.5,-15) ;
\node (B) at (20.2,-7.6) {\text{blocks}}; 
\end{tikzpicture}
\end{center}
\caption{Paths from $R$ to $M_t$ corresponding to minimal generators of $M_t$}
\label{proof_AR1}
\end{figure}
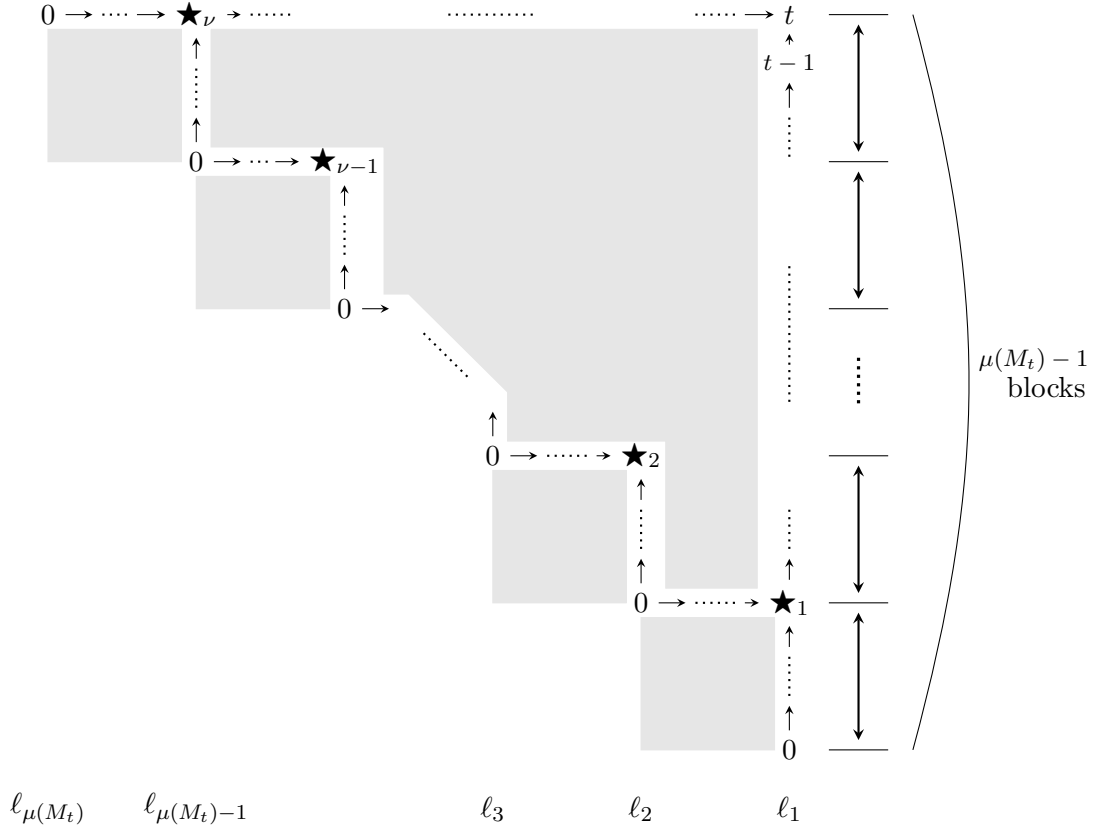

We set $i_s=\mathrm{max}\{\,i_u \mid d_u\neq 0\}$. Then 
we can find the vertex $i_s$ on $\ell_1$. 
From that position, we follow vertices in the left direction, and 
if we arrive at a vertex $0$ we stop there (see Figure~\ref{proof_AR2}). 
Since $M_{i_s}$ is a special CM module, the length of the vertical (resp. horizontal) path from $0$ to $i_s$ 
is $i_s$ (resp. $j_s$), by Theorem~\ref{cyclic_special}.  
By the choice of $i_s$, we have that $\bigstar_1\le i_s$. 
If $\bigstar_1< i_s$, then we obtain Figure~\ref{proof_AR3}, because the $j$-series is a decreasing sequence. 
This contradicts the construction of Figure~\ref{proof_AR1}, and hence it follows that $i_s$ coincides with $\bigstar_1$. 
If $d_s-1\neq 0$, then we can find the vertex $i_s$ on $\ell_2$. 
Then, we again follow vertices from $i_s$ in the left direction until we arrive at $0$. 
Thus, we see that $\bigstar_2=i_s$ by the same argument used in the case that $\bigstar_1$. 
Since $d_si_s\le t$, we can repeat this argument $d_s$ times to find that $\bigstar_1,\cdots,\bigstar_{d_s}$ are all equal to $i_s$. 

{\footnotesize
\begin{figure}[!h]
\begin{tabular}{c}
\begin{minipage}{0.45\hsize}

\begin{center}
\begin{tikzpicture}[>=stealth, scale=0.54]
\node (P0) at (-2,0) {$0$}; 
\node (P1) at (0,0) {$a$}; 
\node (P2) at (6,0) {$i_s$}; 
\node (P3) at (6,-6) {$1$}; 
\node (P4) at (6,-8) {$0$}; 

\draw (-2,0.8) to [out=30, in=150] node [swap,above]{$j_s$} (6,0.8) ;
\draw (6.8,0) to [out=-60, in=60] node [swap,above, xshift=0.25cm]{$i_s$} (6.8,-8) ;

\draw[->] (-1.7,0)--node [swap,above]{$y$} (-0.3,0); 
\draw[->] (0.3,0)--node [swap,above]{$y$} (1.7,0); 
\draw[dotted, thick] (2.2,0)--(3.8,0); 
\draw[->] (4.3,0)--node [swap,above]{$y$} (5.7,0); 

\draw[->] (6,-1.8)--node [swap,above,xshift=0.2cm, yshift=-0.25cm]{$x$} (6,-0.4); 
\draw[dotted, thick] (6,-2.2)--(6,-3.8); 
\draw[->] (6,-5.6)--node [swap,above,xshift=0.2cm, yshift=-0.25cm]{$x$} (6,-4.2); 
\draw[->] (6,-7.6)--node [swap,above,xshift=0.2cm, yshift=-0.25cm]{$x$} (6,-6.5); 

\end{tikzpicture}
\end{center}
\caption{Minimal paths from $R$ to $M_{i_s}$}
\label{proof_AR2}
\end{minipage} 

\begin{minipage}{0.5\hsize}
\begin{center}
\begin{tikzpicture}[>=stealth, scale=0.405]
\node (P0) at (0,0) {$0$}; 
\node (P1) at (6,0) {$i_s$}; 
\node (P3) at (-4,-6) {$0$}; 
\node (P4) at (6,-6) {$\bigstar_1$}; 
\node (P5) at (6,-12) {$0$}; 

\draw[->] (0.3,0)--node [swap,above]{$y$} (1.7,0); 
\draw[dotted, thick] (2.2,0)--(3.8,0); 
\draw[->] (4.3,0)--node [swap,above]{$y$} (P1); 

\draw[->] (-3.7,-6)--node [swap,above]{$y$} (-2.3,-6); 
\draw[dotted, thick] (-1.8,-6)--(-0.2,-6); 
\draw[->] (0.3,-6)--node [swap,above]{$y$} (1.7,-6); 
\draw[dotted, thick] (2.2,-6)--(3.8,-6); 
\draw[->] (4.3,-6)--node [swap,above]{$y$} (P4); 

\draw[->] (6,-1.8)--node [swap,above,xshift=0.2cm, yshift=-0.25cm]{$x$} (P1); 
\draw[dotted, thick] (6,-2.2)--(6,-3.8); 
\draw[->] (6,-5.6)--node [swap,above,xshift=0.2cm, yshift=-0.25cm]{$x$} (6,-4.2); 

\draw[->] (6,-7.8)--node [swap,above,xshift=0.2cm, yshift=-0.25cm]{$x$} (6,-6.5); 
\draw[dotted, thick] (6,-8.2)--(6,-9.8); 
\draw[->] (6,-11.6)--node [swap,above,xshift=0.2cm, yshift=-0.25cm]{$x$} (6,-10.2); 
\end{tikzpicture}
\end{center}

\caption{What is happen if $\bigstar_1< i_s$}
\label{proof_AR3}
\end{minipage}
\end{tabular}
\end{figure}}

Next, we set $i_{s^\prime}=\mathrm{max}\{\,i_u \mid d_u\neq 0, \ i_u<i_s\}$. 
Then, we can find the vertex $i_{s^\prime}$ on $\ell_{d_s+1}$. 
By the same argument as above, we see that $\bigstar_{d_s+1}=i_{s^\prime}$. 
Similarly, we see that $\bigstar_{d_s+1}, \cdots, \bigstar_{d_s+d_{s^\prime}}$ are all equal to $i_{s^\prime}$. 

We can repeat these arguments until we arrive at $\ell_{\mu_R(M_t)}$. 
Then, we see that $\nu$ is equal to the sum of all $d_s$ with $d_s\neq0$, 
and hence $\mu_R(M_t)-1=d_1+d_2+\cdots +d_r$. 
\end{proof}

Since $\rme(R)=\rme(M_t)$ and $\mu_R(M_t)\le\rme(M_t)$, 
we may set $\mu_R(M_t)=\rme(R)-s$, where $0\le s\le\rme(R)-1$.  
The next corollary follows immediately from the theorem. 
Using this corollary, we can determine which $M_t$ are Ulrich modules. 
 
\begin{cor}
\label{main_cor}
Let the notation be the same as above. Then, 
\[
\mu_R(M_t)=\rme(R)-s \Longleftrightarrow d_{1,t}+d_{2,t}+\cdots+d_{r,t}=\rme(R)-(s+1) 
\]
for $s=0, 1, \cdots, \rme(R)-1$. 
In particular, an MCM $R$-module $M_t$ is Ulrich if and only if $d_{1,t}+d_{2,t}+\cdots+d_{r,t}=\rme(R)-1$. 
\end{cor}

\begin{ex}
Let $G=\frac{1}{12}(1,7)$ be a cyclic group of order $12$ (see also Examples~\ref{ex12/7} and \ref{ex12/7_AR}).  
In this case, the non-free special CM modules are $M_7, M_2, M_1$, and $\rme(R)=4$. 
Furthermore, we obtain the following division of each subscript into integers appearing in the $i$-series:
\[\begin{array}{rclccrclccrcl}
  11&=&7+2+2&&&7&=&7&&&3&=&2+1 \\
  10&=&7+2+1&&&6&=&2+2+2&&&2&=&2 \\
  9&=&7+2&&&5&=&2+2+1&&&1&=&1 \\
  8&=&7+1&&&4&=&2+2 &&&&&
\end{array}\] 
Therefore, the Ulrich modules are $M_{11}, \,M_{10}, \,M_6,$ and $M_5$. 
The following figure represents paths in the AR quiver that correspond to minimal generators of $M_{10}$. 
{\small
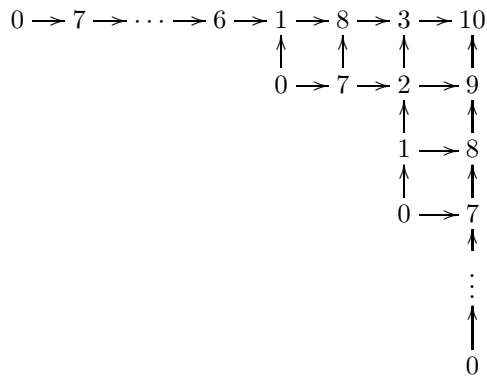
\[\xymatrix@C=12pt@R=12pt{
  0\ar[r]&7\ar[r]&\cdots\ar[r]&6\ar[r]&1\ar[r]&8\ar[r]&3\ar[r]&10 \\
  &&&&0\ar[u]\ar[r]&7\ar[u]\ar[r]&2\ar[r]\ar[u]&9\ar[u] \\
  &&&&&&1\ar[r]\ar[u]&8\ar[u] \\
  &&&&&&0\ar[r]\ar[u]&7\ar[u] \\
  &&&&&&&\vdots\ar[u] \\
  &&&&&&&0\ar[u] 
}\]}
\end{ex}

\begin{rem}
\label{non_cyclic}
The method used in the representation theoretic proof enables us to determine Ulrich modules for other quotient surface singularities. 
For example, see \cite[Example~3.6 and A.5]{Nak2}.
\end{rem}

\medskip

In this manner, we can obtain a characterization of Ulrich modules. 
However, if the order of $G$ is sufficiently large, then a process to obtain the sequence $(d_{1,t}, \cdots, d_{r,t})$ for 
any $t=0, 1, \cdots, n-1$ will be inefficient. 
Therefore, we will demonstrate an additional characterization of Ulrich modules, in terms of the $i$-series. 
First, we decompose each subscript $t=0, 1, \cdots, n-1$ as we did at the beginning of this section: 
\begin{eqnarray}
\label{decomp}
t=d_{1,t}i_1+d_{2,t}i_2+\cdots+d_{r,t}i_r. 
\end{eqnarray}
Then, for each subscript $t=0, 1, \cdots, n-1$ we define a subset of the $i$-series as follows: 
\[
\ttI_t\coloneqq \{ i_s\;|\;d_{s,t}\neq 0 \text{ in the decomposition (\ref{decomp})} \ \text{for}\ s=1,2,\cdots,r-1\}. 
\]
In order to characterize Ulrich modules, we need to determine $\ttI_{n-1}$. 
Since we can decompose $n-1$ as  
\begin{equation}
\label{n-1_decomp}
\begin{array}{rcl}
n-1&=&\alpha_1i_1-i_2-1\\
&=&(\alpha_1-1)i_1+(i_1-i_2)-1 \\
&=&(\alpha_1-1)i_1+(\alpha_2-1)i_2-i_3-1\\
&=&(\alpha_1-1)i_1+(\alpha_2-2)i_2+(i_2-i_3)-1\\
&\vdots& \\
&=&(\alpha_1-1)i_1+(\alpha_2-2)i_2+\cdots+(\alpha_{r-1}-2)i_{r-1}+(\alpha_r-1)i_r-i_{r+1}-1\\ 
&=&(\alpha_1-1)i_1+(\alpha_2-2)i_2+\cdots+(\alpha_{r-1}-2)i_{r-1}+(\alpha_r-2)i_r,  
\end{array}
\end{equation}
we have that
\begin{equation}
\label{def_In}
\ttI_{n-1}=\{i_1\}\cup\{i_s\;|\;\alpha_s>2 \text{ and } 2\le s\le r-1 \}. 
\end{equation}
Since the sum of coefficients is $(\alpha_1-1)+\sum^r_{u=2}(\alpha_u-2)=\alpha_1+\cdots+\alpha_r-2r+1=\rme(R)-1$,  
$M_{n-1}$ is an Ulrich module. 
(This also follows from Proposition~\ref{Ulrich_Rdual}, because we have that $i_r=1$ from the definition of the $i$-series.) 

\medskip

We are now ready to state our main theorem. 

\begin{thm}
\label{main2}
Consider any sequences of the $i$-series $i_{k(1)}>i_{k(2)}>\cdots>i_{k(2b)}$ with $i_{k(2c-1)}\in\ttI_{n-1}$ for all $c=1,\cdots,b$. 
If $t=n-1-\displaystyle\sum^b_{c=1}\big(i_{k(2c-1)}-i_{k(2c)}\big)$ or $t=n-1$, then $M_t$ is an Ulrich module. 
\end{thm}

\begin{proof}
We already know that $M_{n-1}$ is an Ulrich module. 
Thus, we will consider other cases. 

We consider the part of the sequence given by $i_{k(1)}>i_{k(2)}$ with $i_{k(1)}\in\ttI_{n-1}$. 
Using the decomposition (\ref{n-1_decomp}), we rewrite $n-1-(i_{k(1)}-i_{k(2)})$ as
\[\begin{array}{ccl}
\text{( if $k(1)=1$ )}&=&\displaystyle\sum^{k(2)-1}_{v=1}(\alpha_v-2)i_v+(\alpha_{k(2)}-1)i_{k(2)}+\sum^r_{v={k(2)}+1}(\alpha_v-2)i_v.\\
\text{( if $k(1)\neq1$ )}&=&(\alpha_1-1)i_1+\displaystyle\sum^{k(1)-1}_{v=2}(\alpha_v-2)i_v+(\alpha_{k(1)}-3)i_{k(1)} \\
&&+\displaystyle\sum^{k(2)-1}_{v=k(1)+1}(\alpha_v-2)i_v+(\alpha_{k(2)}-1)i_{k(2)}+\sum^r_{v={k(2)}+1}(\alpha_v-2)i_v.
\end{array}\]
In this decomposition, the coefficients of the $i$-series satisfy the conditions given in Lemma~\ref{key_lem1} below, 
and their sum is equal to $\alpha_1+\cdots+\alpha_r-2r+1=\rme(R)-1$. Therefore, $M_{n-1-(i_{k(1)}-i_{k(2)})}$ is 
an Ulrich module, by Corollary~\ref{main_cor}. 
Then, we consider the part of sequence given by $i_{k(3)}>i_{k(4)}$ with $i_{k(3)}\in\ttI_{n-1}$. 
Considering the decomposition of $n-1-(i_{k(1)}-i_{k(2)})-(i_{k(3)}-i_{k(4)})$, 
the same argument implies that $M_t$ is an Ulrich module for $t=n-1-(i_{k(1)}-i_{k(2)})-(i_{k(3)}-i_{k(4)})$.
By repeating these arguments, we obtain the desired conclusion. 
\end{proof}

\medskip

In addition, we can demonstrate the converse of Theorem~\ref{main2}. 

\begin{thm}
\label{main3}
If $M_t$ is an Ulrich module $(t\neq n-1)$, then there exists a sequence 
of the $i$-series $i_{k(1)}>i_{k(2)}>\cdots>i_{k(2b)}$ such that $i_{k(2c-1)}\in\ttI_{n-1}$ for all $c=1,\cdots,b$, and 
\[
t=n-1-\displaystyle\sum^b_{c=1}\big(i_{k(2c-1)}-i_{k(2c)}\big).
\]
\end{thm}

\begin{proof}
Suppose that $M_t$ is an Ulrich module, and let $t=d_1i_1+d_2i_2+\cdots+d_ri_r$, as in Lemma~\ref{key_lem1}. 
Recall that $n-1=(\alpha_1-1)i_1+(\alpha_2-2)i_2+\cdots+(\alpha_r-2)i_r$. Then, we set the integers 
\[
(\varepsilon_1, \varepsilon_2, \cdots, \varepsilon_r)\coloneqq (d_1-(\alpha_1-1), d_2-(\alpha_2-2), \cdots, d_r-(\alpha_r-2)). 
\]
Since $t\neq n-1$, it holds that $(\varepsilon_1, \cdots, \varepsilon_r)\neq (0, \cdots, 0)$. 
By Lemmas~\ref{key_lem2} and \ref{key_lem3}, we can take the sequence of the $i$-series as given in the statement. 
\end{proof} 

To complete the proof of Theorem~\ref{main3}, we require the following lemmas. 

\begin{lem}(\cite[Lemma~1]{Wun1})
\label{key_lem1}
A sequence $(d_1,\cdots,d_r)\in(\mathbb{Z}_{\ge 0})^r$ can be obtained from the description $t=d_1i_1+d_2i_2+\cdots+d_ri_r$ 
for some subscript $t=0, 1, \cdots, n-1$ if and only if the sequence satisfies the following two conditions:  
\begin{itemize}
\item[$\cdot$] $0\le d_u\le\alpha_u-1$ for every $u=1, \cdots, r$. 
\item[$\cdot$] If $d_u=\alpha_u-1$ and $d_v=\alpha_v-1$ $(u<v)$, then there exists $w$ such that $u<w<v$ and $d_w\le\alpha_w-3$. 
\end{itemize}
\end{lem}

\begin{lem}
\label{key_lem2} 
One has that $\varepsilon_1\in\{-1, 0\}$ and $\varepsilon_u\in\{-1, 0, 1\}$ for $u=2, \cdots, r$. 
\end{lem}

\begin{proof}
By Lemma~\ref{key_lem1}, we have that $0\le d_u\le\alpha_u-1$ for any $u\in[1, r]$, 
so that $\varepsilon_1 \le 0$ and $\varepsilon_u\le 1$ for $u\in[2, r]$. 
Since $M_{n-1}$ and $M_t$ are Ulrich modules, we have from Corollary~\ref{main_cor} that $\varepsilon_1+\cdots+\varepsilon_r=0$. 

First, we assume that $\varepsilon_1\le -2$. 
Since $\varepsilon_1+\cdots+\varepsilon_r=0$, there exist at least $-\varepsilon_1\ (\ge2)$ components 
of $(\varepsilon_2, \cdots, \varepsilon_r)$ that satisfy $\varepsilon_u=1$, $u\in[2, r]$. Next, we set $k\coloneqq -\varepsilon_1$, 
and suppose that
\[
\varepsilon_{u_1}=\varepsilon_{u_2}= \cdots=\varepsilon_{u_k}=1 \quad(u_1< u_2<\cdots <u_k) 
\]
are such components. 
By taking the first $k$ components with $\varepsilon_u=1$, 
we may assume that for any $j\in[1, k-1]$ there are no components satisfying $\varepsilon_{u^\prime}=1$ and $u_j<u^\prime<u_{j+1}$. 
Then, by Lemma~\ref{key_lem1} there exists a subscript $v$ such that $u_1<v<u_2$ and $\varepsilon_v\le-1$. 
Thus, there exists a subscript $u_{k+1}$ with $u_k<u_{k+1}$ such that $\varepsilon_{u_{k+1}}=1$, because $\varepsilon_1+\cdots+\varepsilon_r=0$. 
We again apply Lemma~\ref{key_lem1}, to find that there is a subscript $v^\prime$ such that $u_k<v^\prime<u_{k+1}$ and $\varepsilon_{v^\prime}\le-1$. 
These processes can be continued infinitely, but the sequence $(\varepsilon_1, \cdots, \varepsilon_r)$ is finite. 
Thus, we have that $\varepsilon_1\in\{\ -1, 0\}$.  

Next, we assume that there exists $u\in[2, r]$ such that $\varepsilon_u\le-2$. 
As in the above argument, we set $k= -\varepsilon_u \,(\ge 2)$ and $\varepsilon_{u_1}=\varepsilon_{u_2}= \cdots=\varepsilon_{u_k}=1$. 
If $\varepsilon_1=0$, then $d_1=\alpha_1-1$. Thus, the sequence $(d_1, \cdots, d_r)$ is of the form 
\[
(\alpha_1-1, \fbox{\quad A\quad}\,, \alpha_{u_1}-1, \fbox{\quad B \quad}\,, \alpha_{u_2}-1, \cdots). 
\]
By Lemma~\ref{key_lem1}, we can find elements $d_v$ with $d_v\le \alpha_v-3$ in both $A$ and $B$. 
Even if one of these is simply $d_u$ with $\varepsilon_u=d_u-(\alpha_u-2)\le-2$, the other provides us with the conclusion 
that there exists a subscript $u_{k+1}$ with $u_k<u_{k+1}$ such that $\varepsilon_{u_{k+1}}=1$. 
Therefore, we have obtained a contradiction, by the same argument as above. 
If $\varepsilon_1=-1$, then there also exists a subscript $u_{k+1}$ with $u_k<u_{k+1}$ and $\varepsilon_{u_{k+1}}=1$, 
because $\varepsilon_1+\cdots+\varepsilon_r=0$. Thus, we obtain a contradiction in a similar manner. 
As the consequence, we have that $\varepsilon_u\in\{-1, 0, 1\}$ for $u\in [2, r]$. 
\end{proof} 

\begin{lem}
\label{key_lem3}
Let $(\varepsilon_1^\prime, \varepsilon_2^\prime, \cdots, \varepsilon_\ell^\prime)\in\{-1, 1\}^\ell$ be the subsequence of 
$(\varepsilon_1, \varepsilon_2, \cdots, \varepsilon_r)$ obtained by removing all $0$ components from $(\varepsilon_1, \cdots, \varepsilon_r)$. 
Then, $(\varepsilon_1^\prime, \cdots, \varepsilon_\ell^\prime)$ takes an alternating form, as
\[
(-1, +1, -1, +1, \cdots, -1, +1). 
\]
\end{lem}

\begin{proof}
By definition, we have that $\varepsilon_1^\prime+\cdots+\varepsilon_\ell^\prime=0$, and 
the number of $+1$ elements appearing in $(\varepsilon_1^\prime, \cdots, \varepsilon_\ell^\prime)$ coincides with that of $-1$ elements. 

First, we assume that $\varepsilon_1=-1$. Then, $\varepsilon_1=\varepsilon_1^\prime=-1$. 
If the sequence $(\varepsilon_1^\prime, \cdots, \varepsilon_\ell^\prime)$ is not alternating, 
then we can find a section with $+1$ entries appearing consecutively in $(\varepsilon_1^\prime, \cdots, \varepsilon_\ell^\prime)$. 
This contradicts Lemma~\ref{key_lem1}. 

Next, we assume that $\varepsilon_1=0$. Then, $\varepsilon_1\neq \varepsilon_1^\prime$ and $d_1=\alpha_1-1$. 
Thus, we set $\varepsilon_p= \varepsilon_1^\prime\in\{-1, 1\}$. 
If $\varepsilon_p=\varepsilon_1^\prime=1$, then by Lemma~\ref{key_lem1} there exists a subscript $q$ with $1<q<p$ and $\varepsilon_q=-1$. 
This contradicts the choice of $\varepsilon_p$. Therefore, we have that $\varepsilon_p=\varepsilon_1^\prime=-1$. 
If the sequence $(\varepsilon_1^\prime, \cdots, \varepsilon_\ell^\prime)$ is not alternating, 
then we obtain a contradiction by the same argument given in the case of $\varepsilon_1=-1$.  
\end{proof} 

\begin{cor}
\label{main_cor2}
We suppose that $M_t$ is an Ulrich module. 
Then, we have that $n-a\le t\le n-1$. 
Furthermore, $M_{n-1}$ and $M_{n-a}$ are certainly Ulrich modules. 
\end{cor}

\begin{proof}
By Theorem~\ref{main3}, we can describe $t$ as 
$t=n-1-\displaystyle\sum^b_{c=1}\big(i_{k(2c-1)}-i_{k(2c)}\big)$, where 
$i_{k(1)}>i_{k(2)}>\cdots>i_{k(2b)}$ with $i_{k(2c-1)}\in\ttI_{n-1}$ for all $c=1,\cdots,b$. Hence, we have that 
\[
\sum_{c=1}^b\big(i_{k(2c-1)}-i_{k(2c)}\big)=i_{k(1)}-(i_{k(2)}-i_{k(3)})-\cdots -(i_{k(2b-2)}-i_{k(2b-1)})-i_{k(2b)}\le i_1-i_r=a-1. 
\]

Furthermore, $M_{n-1}$ and $M_{n-a}$ are always Ulrich modules, 
because $M_1$ and $M_a$ are special CM modules (see Proposition~\ref{Ulrich_Rdual}). 
\end{proof}

\begin{ex}
Suppose that $G=\frac{1}{158}(1, 57)$. Then, we have that $\frac{158}{57}=[3, 5, 2, 3, 3]$ and 
$i_1=57, i_2=13, i_3=8, i_4=3, i_5=1$, and hence $\ttI_{n-1}=\ttI_{157}=\{i_1, i_2, i_4\}$. 

From the following table and Theorems~\ref{main2} and \ref{main3}, we see that the Ulrich modules are
\[
M_{101}, M_{103}, M_{106}, M_{108}, M_{111}, M_{113}, M_{145}, M_{147}, M_{150}, M_{152}, M_{155} \text{ and } M_{157}.
\] 
 
{\footnotesize
 \begin{center}
 \begin{tabular}{l|l} 
   Sequences $\big(i_{k(1)}>i_{k(2)}>\cdots >i_{k(2b)} \big)$ & $t=n-1-\displaystyle\sum^b_{c=1}\big(i_{k(2c-1)}-i_{k(2c)}\big)$ \\ \hline
   $i_1>i_2 \quad (57>13)$&$157-(57-13)=113$ \\ 
   $i_1>i_3 \quad (57>8)$&$157-(57-8)=108$ \\ 
   $i_1>i_4 \quad (57>3)$&$157-(57-3)=103$ \\ 
   $i_1>i_5 \quad (57>1)$&$157-(57-1)=101$ \\ 
   $i_2>i_3 \quad (13>8)$&$157-(13-8)=152$ \\ 
   $i_2>i_4 \quad (13>3)$&$157-(13-3)=147$ \\ 
   $i_2>i_5 \quad (13>1)$&$157-(13-1)=145$ \\ 
   $i_4>i_5 \quad (3>1)$&$157-(3-1)=155$ \\ 
   $i_1>i_2>i_4>i_5 \quad (57>13>3>1)$&$157-(57-13)-(3-1)=111$ \\ 
   $i_1>i_3>i_4>i_5 \quad (57>8>3>1)$&$157-(57-8)-(3-1)=106$ \\ 
   $i_2>i_3>i_4>i_5 \quad (13>8>3>1)$&$157-(13-8)-(3-1)=150$  
\end{tabular}
\end{center}
 }
\end{ex}

\section{Further topics} 
\label{further_topics}

In this section, we will consider the following question. 

\begin{que}
For a cyclic quotient surface singularity $R$, we fix an integer $1\le m\le \rme(R)$. 
How many indecomposable MCM modules exist that satisfy $\mu_R(M_t)=m$? 
In particular, how many indecomposable Ulrich modules are there?
\end{que}

For simplicity, we denote the number of indecomposable MCM modules $M_t$ satisfying $\mu_R(M_t)=m$ by $\sfN_m$. 
That is, $\sfN_m=\#\{M_t\in\CM(R)\;|\;\mu_R(M_t)=m\}$.

\subsection{Number of minimal generators for each MCM module}
First, we show that there does exist an MCM $R$-module that satisfies $\mu_R(M_t)=m$ for any $m=1, \cdots, \rme(R)$. 
That is, $\sfN_m\ge1$ for any $m=1, \cdots, \rme(R)$. 

\begin{prop}
\label{existence}
For any integer $m=1, \cdots, \rme(R)$, there exists an MCM $R$-module $M_t$ such that $\mu_R(M_t)=m$. 
\end{prop}

\begin{proof}
First, we fix an integer $m=1, \cdots, \rme(R)$, and assume that there exists an MCM $R$-module $M_t$ such that $\mu_R(M_t)=m$.  
Then, for any $\ell =m, m-1, \cdots, 1$ there exists an MCM $R$-module $M_{t^\prime}$ satisfying $\mu_R(M_{t^\prime})=\ell$. 
Indeed, let $t=d_{1,t}i_1+d_{2,t}i_2+\cdots+d_{r,t}i_r$, and set $\widetilde{\ttI}_t=\{i_s \mid d_{s,t}\neq0 \ \text{for $s=1,2,\cdots,r$}\}$. 
Taking $i_u\in \widetilde{\ttI}_t$, we have by Theorem~\ref{main} that $\mu_R(M_{t-i_u})=m-1$. 
Similarly, we take $i_v\in\widetilde{\ttI}_{t-i_u}$, and we have that $\mu_R(M_{t-i_u-i_v})=m-2$. 
Since $d_{1,t}+d_{2,t}+\cdots+d_{r,t}=m-1$ from the hypothesis, we can repeat the above process $m-1$ times. 
Since there exists an Ulrich module, we can apply the above observation to the case that $m=\rme(R)$, thus obtain the desired conclusion.  
\end{proof}

From this result, we obtain the following relations between certain classes of MCM $R$-modules. 

\begin{cor}
\label{cat_ulrich}
Let $R$ be a cyclic quotient surface singularity. Then: 
\begin{itemize}
\item[$(1)$] If $\rme(R)=2$, then $\CM(R)=\mathrm{SCM}(R)=\add(R)\sqcup\mathrm{UCM}(R)$.  
\item[$(2)$] If $\rme(R)=3$, then $\CM(R)=\mathrm{SCM}(R)\sqcup\mathrm{UCM}(R)$. 
\item[$(3)$] If $\rme(R)>3$, then $\CM(R)\supsetneqq\mathrm{SCM}(R)\sqcup\mathrm{UCM}(R)$. 
\end{itemize}
Here, $\mathrm{SCM}(R)$ $($resp. $\mathrm{UCM}(R)$$)$ is the full subcategory of $\CM(R)$ consisting of special $($resp. Ulrich$)$ CM $R$-modules. 
\end{cor}

\begin{rem}
The above results are typical in the case of cyclic quotient surface singularities. 
In fact, we have examples below, when $R$ is not a cyclic quotient surface singularity. 
\begin{itemize}
\item[(1)]
Proposition~\ref{existence} does not hold in a higher dimensional case. For example, we consider the action of $G=\left<\mathrm{diag}(-1, -1, -1)\right>$ on 
$S=\Bbbk[[x, y, z]]$. Then, the invariant subring $R=S^G$ is of finite CM representation type, and finitely many indecomposable MCMs are given by $R, \omega_R$ 
and $\Omega\omega_R$ (cf. \cite{Yo, LW}). Furthermore, we have that $\rme(R)=4$, but $\mu_R(\omega_R)=3$ and $\mu_R(\Omega\omega_R)=8$. 

\item[(2)] 
Corollary~\ref{cat_ulrich} (2) does not hold for non-cyclic cases. 
For example, let $R$ be the invariant subring given in \cite[Example~3.6]{Nak2}. Note that $\rme(R)=3$. 
Then, we can find some indecomposable MCM $R$-modules that are neither special CM modules nor Ulrich modules (see \cite[Example~A.5]{Nak2} and \cite{IW}). 
\end{itemize}
\end{rem}

\subsection{Number of Ulrich modules}

In the previous subsection, we investigated the number $\sfN_m$, and showed that $\sfN_m\ge 1$ for any $m=1, \cdots, \rme(R)$. 
In this subsection, we will focus on $\sfN_{\rme(R)}$, that is, the number of indecomposable Ulrich modules. 
First, we remark that Corollary~\ref{main_cor2} provides us with an upper bound on $\sfN_{\rme(R)}$. 
Namely, we have that $\sfN_{\rme(R)}\le a$. Next, we provide another bound in terms of the number of irreducible exceptional curves. 

\begin{thm}
\label{thm_upper}
Suppose that $R$ is a cyclic quotient surface singularity, whose number of  irreducible exceptional curves 
$($= that of non-free indecomposable special CM modules$)$ is $r$: 

\[\scalebox{0.8}{
\begin{tikzpicture}
\node (A1) at (1,0){$-\alpha_1$};
\node (A2) at (3,0){$-\alpha_2$}; 
\node (Ar) at (7,0){$-\alpha_r$}; 

\draw  [thick] (A1) circle [radius=0.38] ;
\draw  [thick] (A2) circle [radius=0.38] ;
\draw  [thick] (Ar) circle [radius=0.38] ;
\draw (A1)--(A2) ; \draw (A2)--(4.3,0) ; \draw[thick, dotted] (4.5,0)--(5.5,0); \draw (5.7,0)--(Ar) ;
\end{tikzpicture}
}\]

Then, we have that $r\le\sfN_{\rme(R)}\le 2^{r-1}$. 
In particular, we have that $\sfN_{\rme(R)}=r$ for $r=1,2$. 
Furthermore, if $r>2$ then 
$\sfN_{\rme(R)}=2^{r-1}$ holds if and only if $\alpha_u>2$ for all $u=2, \cdots, r-1$, and 
 $\sfN_{\rme(R)}=r$ holds if and only if $\alpha_2=\cdots=\alpha_{r-1}=2$. 
\end{thm} 

\begin{proof}
By Theorems~\ref{main2} and \ref{main3}, $M_t$ is an Ulrich module if and only if $t=n-1$ or $t$ is described by a sequence of the $i$-series satisfying
\[(\clubsuit) \quad  i_{k(1)}>i_{k(2)}>\cdots>i_{k(2b)} \ \text{with}\ i_{k(2c-1)}\in\ttI_{n-1}\ \text{for all}\ c=1,\cdots,b.\]
Note that if we take a different sequence satisfying $(\clubsuit)$, then the corresponding subscripts will also be different, 
because the sequence of integers $(d_1, \cdots, d_r)$ defined in Lemma~\ref{key_lem1} is unique for each subscript $t$. 
Thus, $\sfN_{\rme(R)}-1$ is equal to the number of sequences satisfying the condition $(\clubsuit)$. 
Therefore, we show that the maximal (resp. minimal) number of such sequences is equal to $2^{r-1}-1$ (resp. $r-1$). 

Clearly, to obtain an upper bound on $\sfN_{\rme(R)}$ we should consider the case where the number of elements in $\ttI_{n-1}$ is maximal. 
Therefore, we consider the case that $\ttI_{n-1}=\{i_1,i_2,\cdots,i_{r-1}\}$, which holds if and only if $\alpha_u>2$ for all $u=2, \cdots, r-1$. 
In this situation, we can take any index in $\{i_1,i_2,\cdots,i_{r-1}\}$ as an odd term of sequences. 
Thus, the number of sequences satisfying $(\clubsuit)$ is simply
the number of possible choices of indices in $\{i_1,i_2,\cdots,i_{r-1}\}$ (except for the empty one), which is equal to $2^{r-1}-1$.  
(Note that if we choose an odd number of indices, then we can construct a sequence of the $i$-series satisfying $(\clubsuit)$ by adding $i_r$.) 

Similarly, in order to obtain a lower bound we should consider the case that $\ttI_{n-1}=\{i_1\}$, which holds if and only if $\alpha_2=\cdots=\alpha_{r-1}=2$. 
In this situation, we can easily see that the number of sequence is $r-1$. 
\end{proof}

We remark that there exist two upper bounds, $\sfN_{\rme(R)}\le a$ and $2^{r-1}$, 
and which bound is better depends on the case. 
If $r$ is small, then we can compute $\sfN_{\rme(R)}$ explicitly. 

\begin{ex}
Suppose that $R$ is a cyclic quotient surface singularity, whose dual graph $\sfC$ is given by one of the cases below. 
\begin{itemize}
\item[(1)] The case of 
\[\scalebox{0.8}{
\begin{tikzpicture}
\node (A0) at (0,0) {$\sfC:$} ;
\node (A1) at (1,0){$-\alpha$};
\node (A2) at (3,0){$-\beta$}; 
\node (A3) at (5,0) {$(\alpha, \beta\ge 2)$ .}; 

\draw  [thick] (A1) circle [radius=0.38] ;
\draw  [thick] (A2) circle [radius=0.38] ;

\draw (A1)--(A2) ;
\end{tikzpicture}
}\]
Then, we have that $\sfN_{\rme(R)}=2$. 

\item[(2)]
The case of 
\[\scalebox{0.8}{
\begin{tikzpicture}
\node (A0) at (0,0) {$\sfC:$} ;
\node (A1) at (1,0){$-\alpha$};
\node (A2) at (2.5,0){$-\beta$}; 
\node (A3) at (4,0){$-\gamma$}; 
\node (A4) at (6,0) {$(\alpha, \beta, \gamma\ge 2)$ .}; 

\draw  [thick] (A1) circle [radius=0.38] ;
\draw  [thick] (A2) circle [radius=0.38] ;
\draw  [thick] (A3) circle [radius=0.38] ;
\draw (A1)--(A2) ;
\draw (A2)--(A3) ;
\end{tikzpicture}
}\]
  \begin{itemize}
  \item[(2-1)] If $\beta=2$, then we have that $\sfN_{\rme(R)}=3$. 
  \item[(2-2)] If $\beta>2$, then we have that $\sfN_{\rme(R)}=4$. 
  \end{itemize}
  
\item[(3)]
The case of 
\[\scalebox{0.8}{
\begin{tikzpicture}
\node (A0) at (0,0) {$\sfC:$} ;
\node (A1) at (1,0){$-\alpha$};
\node (A2) at (2.5,0){$-\beta$}; 
\node (A3) at (4,0){$-\gamma$}; 
\node (A4) at (5.5,0){$-\delta$}; 
\node (A5) at (8,0) {$(\alpha, \beta, \gamma, \delta\ge 2)$ .}; 

\draw  [thick] (A1) circle [radius=0.38] ;
\draw  [thick] (A2) circle [radius=0.38] ;
\draw  [thick] (A3) circle [radius=0.38] ;
\draw  [thick] (A4) circle [radius=0.38] ;
\draw (A1)--(A2) ; \draw (A2)--(A3) ; \draw (A3)--(A4) ;

\end{tikzpicture}
}\]
  \begin{itemize}
  \item[(3-1)] If $\beta=2, \gamma=2$, then we have that $\sfN_{\rme(R)}=4$. 
  \item[(3-2)] If $\beta=2, \gamma>2$, then we have that $\sfN_{\rme(R)}=6$. 
  \item[(3-3)] If $\beta>2, \gamma=2$, then we have that $\sfN_{\rme(R)}=6$. 
  \item[(3-4)] If $\beta>2, \gamma>2$, then we have that $\sfN_{\rme(R)}=8$. 
  \end{itemize}
\end{itemize}
\end{ex}

\begin{proof} 
We only demonstrate the case (3-2), as the other cases are similar. 

Let $i_1, \cdots, i_4$ be the $i$-series corresponding to each exceptional curve. 
Since $\beta=2$ and $\gamma>2$, we have that $\ttI_{n-1}=\{i_1, i_3\}$. 
The statement follows from Theorems~\ref{main2} and \ref{main3}, 
because we can take sequences $\{i_1>i_2\}, \{i_1>i_3\}, \{i_1>i_4\}, \{i_3>i_4\}, \{i_1>i_2>i_3>i_4\}$. 
\end{proof}

\subsection{Examples} 
We finish this paper by presenting some typical examples. 
In particular, we completely determine the number $\sfN_m$. 
By Theorem~\ref{main}, we know that special CM modules behave like a ``basis". 
Thus, we can identify each MCM $R$-module $M_t$ with the lattice point $(d_{1,t}, \cdots, d_{r,t})\in\ZZ^r$. 

\begin{ex}
Suppose that $G=\frac{1}{23}(1,6)$ and $R=\Bbbk[[x, y]]^G$. Then, $\frac{23}{6}=4-\frac{1}{6}=[4, 6]$ and $\rme(R)=8$. 
The $i$-series are given by $i_1=6, i_2=1$. 
In this situation, we identify each subscript $t=0,1,\cdots ,22$ with the lattice point $(d_{1,t}, d_{2,t})$. 
For example, because $20=(1+1)+(6+6+6)$, this corresponds to the lattice point $(2,3)$. 
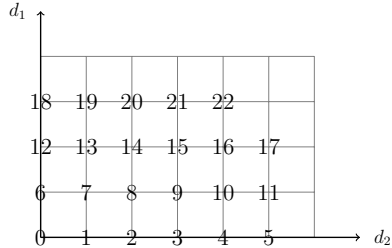
\[\scalebox{0.6}{
\begin{tikzpicture}
\node (P1) at (7.5,0) {{\large $d_2$}};
\node (P2) at (-0.5,5) {{\large $d_1$}};

\draw [step=1,thin, gray] (0,0) grid (6,4);
\draw[->, line width=0.03cm] (0,0)--(0,5);
\draw[->, line width=0.03cm] (0,0)--(7,0);

\node (A0) at (0,0){{\Large{$0$}}};
\node (A1) at (1,0){{\Large{$1$}}};
\node (A2) at (2,0){{\Large{$2$}}};
\node (A3) at (3,0){{\Large{$3$}}};
\node (A4) at (4,0){{\Large{$4$}}};
\node (A5) at (5,0){{\Large{$5$}}};
\node (A6) at (0,1){{\Large{$6$}}};
\node (A7) at (1,1){{\Large{$7$}}};
\node (A8) at (2,1){{\Large{$8$}}};
\node (A9) at (3,1){{\Large{$9$}}};
\node (A10) at (4,1){{\Large{$10$}}};
\node (A11) at (5,1){{\Large{$11$}}};
\node (A12) at (0,2){{\Large{$12$}}};
\node (A13) at (1,2){{\Large{$13$}}};
\node (A14) at (2,2){{\Large{$14$}}};
\node (A15) at (3,2){{\Large{$15$}}};
\node (A16) at (4,2){{\Large{$16$}}};
\node (A17) at (5,2){{\Large{$17$}}};
\node (A18) at (0,3){{\Large{$18$}}};
\node (A19) at (1,3){{\Large{$19$}}};
\node (A20) at (2,3){{\Large{$20$}}};
\node (A21) at (3,3){{\Large{$21$}}};
\node (A22) at (4,3){{\Large{$22$}}}; 

\end{tikzpicture}
}\]

By Theorem~\ref{main}, we have that $d_{1,t}=-d_{2,t}+\mu(M_t)-1$. 
If $\mu_R(M_t)=\rme(R)=8$, then an MCM module whose corresponding lattice point is on the line $d_{1,t}=-d_{2,t}+7$ is an Ulrich module. 
Thus, in this case there are two Ulrich modules, $(M_{17}$ and $M_{22})$. 
Similarly, if $\mu_R(M_t)=7$, then the figure below implies that $\sfN_7=3$.  
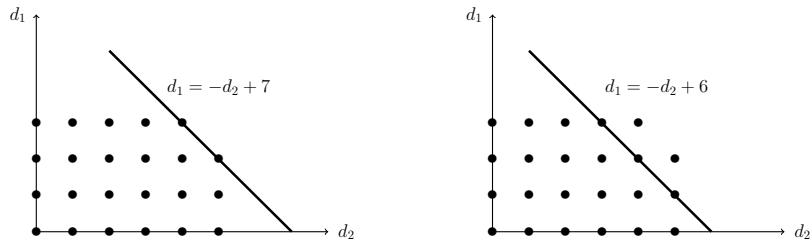
\begin{center}
\begin{tikzpicture} 
\node (fig1) at (0,0) 
{\scalebox{0.48}{
\begin{tikzpicture}
\node (P1) at (8.5,0) {{\Large $d_2$}};
\node (P2) at (-0.5,6) {{\Large $d_1$}};
\node (A0) at (0,0){{\LARGE$\bullet$}};
\node (A1) at (1,0){{\LARGE$\bullet$}};
\node (A2) at (2,0){{\LARGE$\bullet$}};
\node (A3) at (3,0){{\LARGE$\bullet$}};
\node (A4) at (4,0){{\LARGE$\bullet$}};
\node (A5) at (5,0){{\LARGE$\bullet$}};
\node (A6) at (0,1){{\LARGE$\bullet$}};
\node (A7) at (1,1){{\LARGE$\bullet$}};
\node (A8) at (2,1){{\LARGE$\bullet$}};
\node (A9) at (3,1){{\LARGE$\bullet$}};
\node (A10) at (4,1){{\LARGE$\bullet$}};
\node (A11) at (5,1){{\LARGE$\bullet$}};
\node (A12) at (0,2){{\LARGE$\bullet$}};
\node (A13) at (1,2){{\LARGE$\bullet$}};
\node (A14) at (2,2){{\LARGE$\bullet$}};
\node (A15) at (3,2){{\LARGE$\bullet$}};
\node (A16) at (4,2){{\LARGE$\bullet$}};
\node (A17) at (5,2){{\LARGE$\bullet$}};
\node (A18) at (0,3){{\LARGE$\bullet$}};
\node (A19) at (1,3){{\LARGE$\bullet$}};
\node (A20) at (2,3){{\LARGE$\bullet$}};
\node (A21) at (3,3){{\LARGE$\bullet$}};
\node (A22) at (4,3){{\LARGE$\bullet$}};

\draw[->, line width=0.03cm] (0,0)--(0,6);
\draw[->, line width=0.03cm] (0,0)--(8,0);

\draw[line width=0.07cm] (2,5)--(7,0);
\node at (5,4){{\Large$d_1=-d_2+7$}} ;
\end{tikzpicture}
}} ;

\node (fig2) at (6,0) 
{\scalebox{0.48}{
\begin{tikzpicture}
    \node (P1) at (8.5,0) {{\Large $d_2$}};
\node (P2) at (-0.5,6) {{\Large $d_1$}};
\node (A0) at (0,0){{\LARGE$\bullet$}};
\node (A1) at (1,0){{\LARGE$\bullet$}};
\node (A2) at (2,0){{\LARGE$\bullet$}};
\node (A3) at (3,0){{\LARGE$\bullet$}};
\node (A4) at (4,0){{\LARGE$\bullet$}};
\node (A5) at (5,0){{\LARGE$\bullet$}};
\node (A6) at (0,1){{\LARGE$\bullet$}};
\node (A7) at (1,1){{\LARGE$\bullet$}};
\node (A8) at (2,1){{\LARGE$\bullet$}};
\node (A9) at (3,1){{\LARGE$\bullet$}};
\node (A10) at (4,1){{\LARGE$\bullet$}};
\node (A11) at (5,1){{\LARGE$\bullet$}};
\node (A12) at (0,2){{\LARGE$\bullet$}};
\node (A13) at (1,2){{\LARGE$\bullet$}};
\node (A14) at (2,2){{\LARGE$\bullet$}};
\node (A15) at (3,2){{\LARGE$\bullet$}};
\node (A16) at (4,2){{\LARGE$\bullet$}};
\node (A17) at (5,2){{\LARGE$\bullet$}};
\node (A18) at (0,3){{\LARGE$\bullet$}};
\node (A19) at (1,3){{\LARGE$\bullet$}};
\node (A20) at (2,3){{\LARGE$\bullet$}};
\node (A21) at (3,3){{\LARGE$\bullet$}};
\node (A22) at (4,3){{\LARGE$\bullet$}};

\draw[->, line width=0.03cm] (0,0)--(0,6);
\draw[->, line width=0.03cm] (0,0)--(8,0);

\draw[line width=0.07cm] (1,5)--(6,0);
\node at (4.5,4){{\Large$d_1=-d_2+6$}} ;
\end{tikzpicture}
}} ;
\end{tikzpicture}
\end{center}
\end{ex}
In this way, we can compute the number $\sfN_m$. In general, we have the following.

\begin{ex}
\label{count_number1}
Take integers $\alpha, \beta\ge 2$, 
and consider $G=\frac{1}{n}(1,a)$ satisfying $n/a=\alpha-\frac{1}{\beta}$. 
Then, $R=S^G$ is the cyclic quotient surface singularity whose dual graph is 
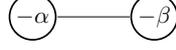
\[\scalebox{0.8}{
\begin{tikzpicture}

\node (A1) at (0,0){$-\alpha$};
\node (A2) at (2,0){$-\beta$}; 
\draw  [thick] (A1) circle [radius=0.38] ;
\draw  [thick] (A2) circle [radius=0.38] ;
\draw (A1)--(A2) ;
\end{tikzpicture}
}\]
and $\rme(R)=\alpha+\beta-2$. The $i$-series are given by $i_1=\beta, \; i_2=1$. 

If $\alpha\le\beta$, then we have the following: 
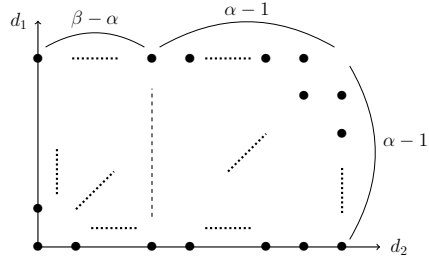
\[\scalebox{0.5}{
\begin{tikzpicture}
\node (P1) at (9.5,0) {{\Large $d_2$}};
\node (P2) at (-0.5,6) {{\Large $d_1$}};
\node  at (0,0){{\LARGE$\bullet$}};
\node  at (1,0){{\LARGE$\bullet$}};
\node  at (3,0){{\LARGE$\bullet$}};
\node  at (4,0){{\LARGE$\bullet$}};
\node  at (6,0){{\LARGE$\bullet$}};
\node  at (7,0){{\LARGE$\bullet$}};
\node  at (8,0){{\LARGE$\bullet$}};
\node  at (0,1){{\LARGE$\bullet$}};
\node  at (0,5){{\LARGE$\bullet$}};
\node  at (3,5){{\LARGE$\bullet$}};
\node  at (4,5){{\LARGE$\bullet$}};
\node  at (6,5){{\LARGE$\bullet$}};
\node  at (7,4){{\LARGE$\bullet$}};
\node  at (7,5){{\LARGE$\bullet$}};
\node  at (8,3){{\LARGE$\bullet$}};
\node  at (8,4){{\LARGE$\bullet$}};

\draw[->, line width=0.03cm] (0,0)--(0,6);
\draw[->, line width=0.03cm] (0,0)--(9,0);
\draw[dotted, ultra thick] (1.4,0.5)--(2.6,0.5) ;
\draw[dotted, ultra thick] (0.9,5)--(2.1,5) ;
\draw[dotted, ultra thick] (4.4,0.5)--(5.6,0.5) ;
\draw[dotted, ultra thick] (4.4,5)--(5.6,5) ;
\draw[dotted, ultra thick] (0.5,1.4)--(0.5, 2.6) ;
\draw[dotted, ultra thick] (1,1)--(2, 2) ;
\draw[dotted, ultra thick] (8,0.9)--(8,2.1) ;
\draw[dotted, ultra thick] (5,2)--(6, 3) ;

\draw[dashed, thick] (3,0.8)--(3, 4.2) ;
\draw (8.2,0.2) to [out=60, in=-60] node [ swap,above, xshift=0.8cm]{{\Large$\alpha-1$}} (8.2,4.8) ;
\draw (3.2,5.3) to [out=30, in=150] node [ swap,above]{{\Large$\alpha-1$}} (7.8,5.3) ;
\draw (0.2,5.3) to [out=30, in=150] node [ swap,above]{{\Large$\beta-\alpha$}} (2.8,5.3) ;
\end{tikzpicture}
}\]  

\[\begin{array}{lclclclclcl}
\sfN_{\alpha+\beta-2}&=&2,&&\sfN_{\beta-1}&=&\alpha,&&\sfN_{\alpha-2}&=&\alpha-2, \\
\sfN_{\alpha+\beta-3}&=&3,&&&\vdots&&&&\vdots& \\
&\vdots&&&\sfN_{\alpha}&=&\alpha,&&\sfN_{2}&=&2,\\
\sfN_{\beta}&=&\alpha,&&\sfN_{\alpha-1}&=&\alpha-1,&&\sfN_{1}&=&1. 
\end{array}\]

The case with $\alpha\ge\beta$ is similar. (Replace $\alpha$ by $\beta$ and vice versa.) 
\end{ex}

We can also determine $\sfN_m$ in the following situation.  

\begin{ex}
\label{count_number2}
Consider a cyclic quotient surface singularity $R$ whose dual graph is 
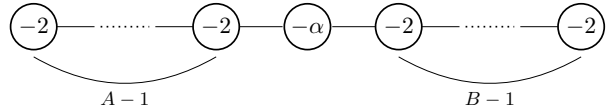
\[\scalebox{0.8}{
\begin{tikzpicture}
\node (A1) at (0,0){$-2$};
\node (A2) at (3,0){$-2$}; 
\node(A3) at (4.5,0){$-\alpha$};
\node (A4) at (6,0){$-2$}; 
\node (A5) at (9,0){$-2$}; 
\draw  [thick] (A1) circle [radius=0.38] ;
\draw  [thick] (A2) circle [radius=0.38] ;
\draw  [thick] (A3) circle [radius=0.38] ;
\draw  [thick] (A4) circle [radius=0.38] ;
\draw  [thick] (A5) circle [radius=0.38] ;

\draw (A1)--(1,0) ;
\draw[thick, dotted] (1.1,0)--(1.9,0) ;
\draw (2,0)--(A2)--(A3)--(A4)--(7,0) ;
\draw[thick, dotted] (7.1,0)--(7.9,0) ; 
\draw (8,0)--(A5) ;

\draw (0,-0.5) to [out=-30, in=210] node [swap,below]{{\footnotesize$A-1$}} (3,-0.5) ;
\draw (6,-0.5) to [out=-30, in=210] node [swap,below]{{\footnotesize$B-1$}} (9,-0.5) ;
\end{tikzpicture}
}\]
where $\alpha\ge 2$ and $A, B\ge 1$. Then, $\rme(R)=\alpha$, and we have the following: 
\[\begin{array}{lclcclcl}
\sfN_{\alpha}&=&AB, &&& \sfN_{3}&=&AB,  \\
\sfN_{\alpha-1}&=&AB, &&& \sfN_{2}&=&A+B-1, \\
&\vdots&&&& \sfN_{1}&=&1. \\
\sfN_{4}&=&AB, &&&&&
\end{array}\] 
\end{ex}

\begin{rem}
If we can obtain $\sfN_m$ for all $m=1, \cdots, \rme(R)$, as in Example~\ref{count_number1} and \ref{count_number2}, 
then we can compute the Hilbert-Kunz multiplicity $\rme_{\text{HK}}(R)$. 
This is a numerical invariant in positive characteristic commutative algebras, 
and can be obtained by the formula $\rme_{\text{HK}}(R)=\frac{1}{n}\sum^n_{t=0}\mu_R(M_t)$ 
in our situation (see \cite[Appendix]{Nak2}). 
Thus, we let $R$ be defined as in Example~\ref{count_number1} (resp. Example~\ref{count_number2}). 
Then, we have that $\rme_{\text{HK}}(R)=\frac{1}{2n}\{(\alpha\beta-2)(\alpha+\beta)+2\}$ 
(resp. $\rme_{\text{HK}}(R)=\frac{1}{2n}\{AB(\alpha-2)(\alpha+3)+4(A+B)-2\}$). 
\end{rem}

\medskip

\subsection*{Acknowledgements}
The authors would like to thank Professor Mitsuyasu Hashimoto and Professor Yukari Ito for careful reading 
the previous version of this paper and for their valuable comments. 
The authors also thank the referee for providing us with many valuable suggestions. 
The authors would like to thank Editage (www.editage.jp)  for English language editing. 

The first author is supported by Grant-in-Aid for JSPS Fellows 26-422. 
The second author is supported by JSPS Grant-in-Aid for Scientific Research (C) 25400050. 


\end{document}